\documentclass[10pt]{article}
\usepackage{geometry}
\usepackage{graphicx}
\usepackage{amsmath,amssymb,amsthm}
\usepackage[utf8]{inputenc}
\usepackage[greek,english]{babel}
\usepackage{enumerate}
\usepackage{mathrsfs}
\usepackage{hyperref}
\usepackage[active]{srcltx}
\numberwithin{equation}{section}
\usepackage[T1]{fontenc}
\usepackage{color}
\newtheorem{theo}{Theorem}
\newtheorem{lem}{Lemma}[section]
\newtheorem{defi}{Definition}[section]

\newtheorem{prop}{Proposition}[section]
\newtheorem{rmk}{Remark}[section]

\newcommand{\R}{\mathbb{R}}
\newcommand{\Z}{\mathbb{Z}}
\newcommand{\N}{\mathbb{N}}
\newcommand{\Q}{\mathbb{Q}}

\newcommand{\bqs}{\begin{equation*}}
\newcommand{\eqs}{\end{equation*}}
\newcommand{\bqq}{\begin{equation}}
\newcommand{\eqq}{\end{equation}}

\DeclareMathAlphabet\mathbfcal{OMS}{cmsy}{b}{n}

\date{}

\begin{document}
\title{Admissible speeds in spatially periodic bistable reaction-diffusion equations}
\author{Weiwei Ding\footnote{School of Mathematical Sciences, South China Normal University, Guangzhou 510631, China},  Thomas Giletti\footnote{IECL UMR 7502, University of Lorraine, B.P. 70239, F-54506 Vandoeuvre-l\`{e}s-Nancy Cedex}}
\maketitle

\begin{abstract}
Spatially periodic reaction-diffusion equations typically admit pulsating waves which describe the transition from one steady state to another. Due to the heterogeneity, in general such an equation is not invariant by rotation and therefore the speed of the pulsating wave may a priori depend on its direction. However, little is actually known in the literature about whether it truly does: surprisingly, it is even known in the one-dimensional monostable Fisher-KPP case that the speed is the same in the opposite directions despite the lack of symmetry. Here we investigate this issue in the bistable case and show that the set of admissible speeds is actually rather large, which means that the shape of propagation may indeed be asymmetrical. More precisely, we show in any spatial dimension that one can choose an arbitrary large number of directions, and find a spatially periodic bistable type equation to achieve any combination of speeds in those directions, provided those speeds have the same sign. In particular, in spatial dimension 1 and unlike the Fisher-KPP case, any pair of (either nonnegative or nonpositive) rightward and leftward wave speeds is admissible.
We also show that these variations in the speeds of bistable pulsating waves lead to strongly asymmetrical situations in the multistable equations. 
\end{abstract}

\section{Introduction}

In this work, we consider a spatially periodic reaction-diffusion equation of the form
\begin{equation}\label{eq:main}
\partial_t u  = \Delta u + f(x,u), \quad t \in \R \, , \ x \in \R^d.
\end{equation}
When it is of the bistable type (see below for a more precise definition), there exists a pulsating wave in each direction~$e \in \mathbb{S}^{d-1}$, namely an entire in time solution which moves with a constant speed through the domain and typically describes the spatio-temporal transition from one stable steady state to another. We will denote by $c^*$ the speed of those pulsating waves, and this defines a function
$$e \in \mathbb{S}^{d-1} \mapsto c^* (e).$$
In the homogeneous case, i.e. $f (x,u) = f(u)$, then the equation is invariant by rotation and therefore the function $c^* (e)$ is actually constant on the unit sphere. However, in the general case then the function $c^* (e)$ may a priori depend on $e \in \mathbb{S}^{d-1}$ in a non-trivial way; we refer to \cite{ducasse,dr} for some related examples in periodic domains with holes. This may in turn have significant consequences on the asymptotic shape of propagation of solutions of the Cauchy problem~\eqref{eq:main}. The purpose of the present paper is to describe the set of admissible speeds, i.e. the set of functions that are achievable when varying the heterogeneous reaction term~$f$.\\

Let us first define more precisely the class of reaction-diffusion equations which we will consider.
\begin{defi}\label{bistable}
We say that \eqref{eq:main} is of the spatially periodic \textbf{bistable} type if $f \in C^1 (\R^{d+1} ; \R) $ and:
\begin{itemize}
\item there exist $L_1, \cdots, L_d >0$ such that, for any integers $k_1, \cdots, k_d$,
		$$ f (x + k L  , u) \equiv f (x, u ) ,$$
 where $k L := (k_1 L_1 , \cdots , k_d L_d)$ and the vector $L := (L_1, \cdots , L_d)$ is the \textbf{period} of~\eqref{eq:main};
\item the constants 0 and 1 are linearly stable steady states of \eqref{eq:main} with respect to spatially $L$-periodic solutions;
\item any other $L$-periodic steady state of \eqref{eq:main} between 0 and 1 is linearly unstable.
\end{itemize}
\end{defi}

Here, an $L$-periodic steady state $\bar{u}$ of \eqref{eq:main} is said to be linearly stable ({\it resp.} linearly unstable) if $\lambda_1(\bar{u})<0$ ({\it resp.} $\lambda_1(\bar{u})>0$), where $\lambda_1(\bar{u})$ is the principal eigenvalue of the problem 
\begin{equation*}%\label{p-eigen-hd}
\Delta\psi+\partial_uf(x,\bar{u})\psi=\lambda\psi\hbox{ in }\R^d,\quad \psi>0\hbox{ in }\R^d,\quad \psi\hbox{ is }L\hbox{-periodic}.
\end{equation*}

Next, we recall the definition of a pulsating wave:
\begin{defi}\label{pulsating-tw}
A \textbf{pulsating wave} connecting 0 and 1 is an entire in time solution $U(t,x)$ of~\eqref{eq:main} of the type
\begin{equation}\label{formula-tw}
U (t,x) = \Phi ( x \cdot e - ct ,x),
\end{equation}
where $c \neq 0$ and the function $\Phi$ is periodic in its second variable, and satisfies
$$\Phi (- \infty , \cdot) = 1 , \quad  \Phi ( + \infty , \cdot) = 0,$$
where both convergences are understood to be uniform with respect to the second variable. We call $c \in \R^*$ the speed of the wave and $e \in \mathbb{S}^{d-1}$ its direction. 

Furthermore, we say that $U (t,x) = U (x)$ is a \textbf{stationary pulsating wave}, or a pulsating wave with speed $c=0$, if it solves~\eqref{eq:main} and satisfies that  $U (x) \to 1$ ({\it resp.}~0) as $x \cdot e \to -\infty$ (\it{resp.} $+\infty$).
\end{defi} 
This notion has been introduced by Xin~\cite{xin,xin2} and Shigesada, Kawasaki, Teramoto~\cite{skt}.    
It is the extension of a more classical notion of a traveling wave in the homogeneous case (see for instance~\cite{aw,fm1,fm2}).

We notice here that in the case $c\neq 0$, formula \eqref{formula-tw} can be written as 
$$\Phi(\xi,x)=U\left(\frac{x\cdot e - \xi}{c},x\right)\,\,\hbox{ for all }\,\, (\xi,x)\in\R\times\R^d,  $$
while the periodicity of $\Phi(\xi,x)$ in $x$ means that there exist $L=(L_1,\cdots L_d)$ with $L_i>0$ for all $1\leq i\leq d$ such that for any $k\in\Z^d$, 
\begin{equation}\label{formula-period}
U\left(t+ \frac{kL\cdot e}{c},x+kL\right)=U\left(t,x\right) \,\,\hbox{ for all }\,\, (t,x)\in\R\times\R^d.
\end{equation}
\begin{rmk}\label{rmk:c=0}
We point out that the change of variables $(t,x) \mapsto (x \cdot e - ct ,x)$ is revertible only when $c \neq 0$, which is why the case when $c=0$ must be defined separately. In particular, when $c \neq 0$, one can check that the function $\Phi$ satisfies a (degenerate) elliptic equation and thus possesses some regularity.
\end{rmk}
It is known~\cite{ducrot,fz,gr} that, when~\eqref{eq:main} is of the bistable type in the above sense, then the following existence and uniqueness result holds true.
\begin{theo}\label{th:tw}
Assume that \eqref{eq:main} is bistable in the sense of Definition~\ref{bistable}. Then for each direction $e \in \mathbb{S}^{d-1}$, there exists a unique speed $c^* (e)$ such that there exists a pulsating wave $U(t,x)$ with speed $c^* (e)$ in direction $e$. If $c^* (e)\neq 0$, then
\begin{equation}\label{sign-property}
{\rm sign}(c^* (e)) = {\rm sign}\left(\int_0^1\int_{(0,L_1)\times\cdots\times(0,L_d)} f(x,u)dxdu\right).
\end{equation}
Furthermore, if there exists $\delta>0$ such that 
\begin{equation*}
s\mapsto f(x,s) \hbox{ is nonincreasing in } [0,\delta]  \hbox{ and in } [1-\delta,1],
\end{equation*}
then the wave $U(t,x)$ is increasing ({\it resp.} decreasing) with respect to $t\in\R$ if $c^* (e)>0$ ({\it resp.} if $c^* (e)<0$), and it is unique up to time shifts if $c^* (e)\neq 0$. 
\end{theo}
More precisely, we refer to \cite[Theorems 1.5, 1.6]{ducrot} for the existence of the pulsating waves and the uniqueness and sign property of the wave speeds. In particular, the sign property~\eqref{sign-property} follows from a simple integration argument on the equation satisfied by $\Phi$. It should be pointed out that this property is only valid under the condition $c^* (e)\neq 0$ (see~\cite{xin2} for an example in dimension $1$ where there exists a nonstationary pulsating wave but the integral of $f$ is equal to $0$). 
The monotonicity and uniqueness results are consequences of~\cite[Theorems 1.11, 1.14]{bh2}  which dealt with a larger class of waves.  We also refer to \cite{dhz1,dhz2} for sufficient conditions on $f$ ensuring the bistable structure of \eqref{eq:main} and more qualitative properties of the pulsating waves in dimension 1. In particular, the one-dimensional nonstationary pulsating wave is globally stable and unique up to time shifts, while the stationary pulsating wave may be neither stable nor unique.

\subsection{Admissible wave speeds and main results} 

The goal of this paper is to describe the set of admissible speeds, i.e. the set of functions
$$\mathbfcal{A} := \{ c : \mathbb{S}^{d-1} \mapsto \R \, : \ \mbox{there exists a spatially periodic bistable equation such that $c^* \equiv c$} \}.$$
In the one-dimensional case $d= 1$, then $\mathbb{S}^{d-1}$ is reduced to the two points $\pm 1$ and $\mathbfcal{A}$ is a subset of $\R^2$. Moreover, it immediately follows from~\eqref{sign-property} that $\mathbfcal{A}$ is also a subset of $\R_+^2 \cup \R_-^2$. Our first main result shows that this is optimal:
\begin{theo}\label{arb-speed-1D}
The set of admissible speeds in spatial dimension~$d= 1$ is
$$\mathbfcal{A} = \R_+^2 \cup \R_-^2 .$$
More precisely, for any $c_L \geq 0$ ({\it resp.} $c_L \leq 0$) and $c_R \geq 0$ ({\it resp.} $c_R \leq 0$), there exists a spatially periodic equation in the sense of Definition~\ref{bistable} which admits a leftward pulsating wave with speed $c_L$, and a rightward pulsating wave with speed $c_R$.
\end{theo}
In other words, both speeds can be chosen almost independently in the one-dimensional case. This was far from obvious a priori. Indeed, when \eqref{eq:main} is of the spatially periodic and monostable Fisher-KPP case, i.e. when $0$ and $1$ are respectively linearly unstable and stable, and $f(x,s)/s$ is a decreasing function of $s >0$, then it is well known (see, e.g., \cite{bh1}) that the minimal wave speed of pulsating waves admits a variational characterization in terms of a family of eigenvalues of an elliptic operator. Even if \eqref{eq:main} is not symmetric, it follows from this characterization that the Fisher-KPP wave speed is the same in the left and right directions in dimension~1 (see for instance~\cite[equation~(10) and Theorem~2.1]{nadin}). Our result shows that the situation is completely different in the bistable case.\\

In higher dimension, we will prove the following result:
\begin{theo}\label{theo-hd}
Let $d\geq 2$. For any $N \geq 2$, any finite set $(\zeta_1, \cdots , \zeta_N)$ of different directions in $\mathbb{S}^{d-1} \cap \Q^d$ and any finite set $(c_1,\cdots  ,c_N)$ of $\R_+^N \cup \R_-^N$, there exists $c \in \mathbfcal{A}$ such that $$c_{|(\zeta_1,\cdots , \zeta_N)} = (c_1,\cdots ,c_N).$$
\end{theo}
According to this theorem, the speeds can be chosen independently in any number of directions whose coordinates are rational. The only condition is that those speeds should have the same sign (in the large sense), which again is a necessary condition by~\eqref{sign-property}. As a matter of fact, we will prove a slightly more general result where the directions are only assumed to be rationally proportional (see Theorem~\ref{main-prop-hd} in Section~\ref{sec:dd}). Here we wrote this simpler statement for readability.

As we mentioned earlier, another example of a situation where the speed truly depends on the direction was exhibited in~\cite{ducasse} using periodic domains. Here we go further in the understanding of this dependence: indeed the set $\mathbb{S}^{d-1} \cap \Q^d$ is dense in $\mathbb{S}^{d-1}$ (see Lemma~\ref{dense} for details), so that Theorem~\ref{theo-hd} can be understood as some kind of density result of $\mathbfcal{A}$ in the set of functions from $\mathbb{S}^{d-1}$ to $\R$. For instance, for any function $c$, one can find a sequence of bistable type reaction terms $(f_N)_{N \in \mathbb{N}}$ such that $c^*_N (\zeta) = c (\zeta)$ for $\zeta \in S^{d-1} \cap \Q^d$ and $N$ large (depending on $\zeta$), i.e. $\mathbfcal{A}$ is indeed dense with respect to the topology of pointwise convergence in a dense subset of $\mathbb{S}^{d-1}$. Unfortunately, in our construction then either the period $L$ goes to infinity or the reaction term $f_N$ becomes singular as $N \to +\infty$. Still, Theorem~\ref{theo-hd} leads us to conjecture that any continuous speed function may be admissible.\\

The possibility of such arbitrarily distinct speeds may have important consequences on the shape of propagation of solutions. Indeed, let us consider the typical case of a solution of \eqref{eq:main} whose initial data is compactly supported. Provided that 
\begin{equation}\label{speed:positive}
\min_{e \in \mathbb{S}^{d-1}} c^* (e) >0,
\end{equation}
and
$$\lim_{t \to +\infty} u (t,x) = 1\,\,\hbox{ locally uniformly in }\, x\in\R^d,$$
it is expected that the solution spreads in each direction~$e \in \mathbb{S}^{d-1}$ with the speed
\begin{equation}\label{eq:fg}
w^* (e) := \min_{e ' \in \mathbb{S}^{d-1}, e'\cdot e>0} \frac{ c^* (e')}{e'\cdot e}.
\end{equation}
By spreading speed, we mean that
$$\forall\, c > w^* (e), \quad \lim_{t \to +\infty} u (t, x + cte) = 0,$$
$$\forall \,0 \leq c < w^* (e), \quad \lim_{t \to +\infty} u (t, x+ cte) = 1,$$
where both limits are understood to be locally uniform with respect to $x \in \R^d$. This is often referred to as the Freidlin-Gartner formula, since they proved it in the monostable Fisher-KPP case~\cite{fg}; we also mention~\cite{rossi} for a proof in a bistable particular case.

Denoting by $(e_1,\cdots ,e_d)$ the standard basis, it follows from Theorem~\ref{theo-hd} that there is a situation where
$$w^* (e_1) < c^* (e_1).$$
Moreover, one may easily check from our proof that one can simultaneously make the equation~\eqref{eq:main} symmetrical with respect to the other directions $e_2$, ... , $e_d$, so that the minimum in the definition of $w^* (e_1)$ is reached simultaneously in two distinct directions $e$ and $e'$ with $(e - e') \cdot e_1 = 0$. In that case, the distance grows linearly in time between the respective level sets of the pulsating wave with direction~$e_1$ and the solution with compactly supported initial data, a new phenomenon which does not occur in a homogeneous equation. Furthermore, an angle appears in the large time behavior of the solution, more precisely in the asymptotic set of spreading $ \{ r\xi \, : \ 0 \leq r \leq w^* (\xi) \}$, and in particular the latter is not smooth. This also suggests that some spatially periodic bistable equations admit a new type of conical waves, where the superlevel sets (i.e., under \eqref{speed:positive}, the zones where the solution is closer to the invading steady state) are convex (contrary to the situation in~\cite{hmr,nt}).

As we will discuss below, these variations in the pulsating wave speeds also lead to strongly asymmetrical situations in the multistable framework.

\subsection{Asymmetric propagating terraces in the multistable case}

In this subsection, we consider a more general situation where \eqref{eq:main} may admit more than two stable steady states. On the other hand, for simplicity we restrict ourselves to one spatial dimension, i.e. $d=1$, and only briefly discuss the higher dimension case. Let us thus introduce the following notion:
\begin{defi}\label{def:multistable}
We say that \eqref{eq:main} is of the spatially periodic \textbf{multistable} type if $f \in C^1 (\R^{2} ; \R) $ and:
\begin{itemize}
\item there exists $L>0$ such that $ f (x + L, u) \equiv f (x, u )$;
\item there exists a finite sequence $p \equiv p_0 > p_ 1 > \cdots > p_I \equiv 0$ of linearly stable $L$-periodic steady states;
\item for each $1 \leq k \leq I$, there exists a pulsating wave connecting $p_k$ and $p_{k-1}$.
\end{itemize}
\end{defi}
We point out that the definition of a pulsating wave connecting $p_{k}$ to $p_{k-1}$ is simply obtained by replacing 0 and 1 by $p_k$ and $p_{k-1}$ in Definition~\ref{pulsating-tw}. Provided that its speed is non-zero, then such a pulsating wave is unique up to time shifts by the same argument as in the bistable case; this is also a particular case of Lemma~\ref{terrace-minimal} below.

Let us mention that other definitions of multistability are used in the literature. For instance, the above is not the same as the definition used by one of the authors in~\cite{gr}. Our definition typically includes the stacking of several bistable type reaction terms with the same period. We stress that such notions of multistability should be handled carefully: for instance, we point out that Definition~\ref{def:multistable} allows for the existence of stable steady states which do not belong to the sequence $(p_i)_{0 \leq i \leq I}$. 

In such a situation, the notion of a single pulsating wave may no longer be enough to describe the transition between the extremal steady states 0 and $p$. It must be replaced by the following notion of a propagating terrace, or stacked fronts:
\begin{defi}\label{terrace}
A {\bf propagating terrace} connecting $0$ and $p$ in the right ({\it resp.} left) direction is a pair of finite sequences $(q_k)_{0\leq k\leq N}$ and $(U_{k},c_{k})_{1\leq k\leq N}$ such that
\begin{itemize}
\item[$(i)$] each $q_k$ is a nonnegative periodic steady state of \eqref{eq:main} satisfying 
$$p\equiv q_0>q_1>\cdots> q_{N}\equiv 0;$$
\item[$(ii)$] for each $1\leq k\leq N$, the function $U_{k}(t,x)$ is a rightward ({\it resp.} leftward) pulsating wave of \eqref{eq:main} connecting $q_k$ and $q_{k-1}$ with speed $c_{k}\in\R$;
\vskip 3pt
\item[$(iii)$] the sequence $(c_k)_{1\leq k\leq N}$ satisfies $c_1\leq c_2\leq \cdots \leq c_N$. 
\end{itemize}
We denote such a propagating terrace by $\mathcal{T}:= ((U_{k},c_{k})_{1\leq k\leq N}, (q_k)_{0\leq k\leq N})$ and call $(q_k)_{0\leq k\leq N}$ the {\bf platforms} of $\mathcal{T}$. 
\end{defi}
The notion of a propagating terrace was introduced in the spatially periodic case by one of the authors in~\cite{dgm}, under the additional assumption that all speeds are positive. It was then extended to a situation where the speeds may have different signs in~\cite{gm}, and in~\cite{gr} to the higher dimensional case. It also appeared in the much earlier works of Fife and McLeod~\cite{fm1,fm2} under the name of ``minimal decomposition'', but only in the homogenenous case where ODE technics are available. We mention that the convergence to propagating terraces from various types of initial data was extensively studied recently in the homogeneous and time-periodic cases (see, e.g., \cite{dm2,dum2,po1,po2}). 

As we will recall in Section~\ref{sec:terraces}, it was shown in~\cite{gm} that terraces exist in a general spatially periodic framework, including the case when \eqref{eq:main} is multistable in the sense of Definition~\ref{def:multistable} (see Lemma~\ref{terrace-minimal} below). Moreover, the propagating terrace is also unique provided that all its speeds are non-zero; by uniqueness of the propagating terrace, we mean that all terraces share the same platforms and that between two consecutive platforms the pulsating wave is unique up to time shift. However, as was discussed in~\cite{gm}, in the case when some speed is equal to 0, even the number of platforms $N+1$ may no longer be unique.

In the following theorem, we exhibit various examples to show that the shapes of the terraces (i.e. not only the speeds of pulsating fronts but also the intermediate platforms or even their number) may be completely different in two opposite directions.

\begin{theo}\label{nonsysm-terrace}
Let $d=1$ and $N\geq 1$ be an arbitrary integer. The following statements hold true.
\begin{itemize} 
\item[$(i)$]
There exists a spatially periodic reaction $f(x,u)$ such that \eqref{eq:main} has a unique rightward propagating terrace $((U_{R,k},c_{R,k})_{1\leq k\leq N},(q_{R,k})_{0\leq k\leq N})$ connecting $0$ and $1$, and a unique leftward propagating terrace  $((U_{L,k},c_{L,k})_{1\leq k\leq N},(q_{L,k})_{0\leq k\leq N})$ connecting $0$ and $1$ such that 
$$q_{R,k}=q_{L,k}\,\,\hbox{ for each }\, 0\leq k\leq N, \quad\hbox{and}\quad   c_{R,k} \neq  c_{L,k} \,\hbox{ for each }\,  1\leq k\leq N. $$

\item[$(ii)$] There exists a spatially periodic reaction $f(x,u)$ such that \eqref{eq:main} has a unique rightward propagating terrace $((U_{R,k},c_{R,k})_{1\leq k\leq N},(q_{R,k})_{0\leq k\leq N})$ connecting $0$ and $1$, and a unique leftward propagating terrace  $((U_{L,k},c_{L,k})_{1\leq k\leq N},(q_{L,k})_{0\leq k\leq N})$ connecting $0$ and $1$ such that 
$$\{q_{R,k}:\, k=1,2,\cdots,N-1 \} \cap \{q_{L,k}:\, k=1,2,\cdots,N-1 \} = \emptyset.$$

\item[$(iii)$] There exists a spatially periodic reaction $f(x,u)$ such that \eqref{eq:main} has a unique leftward  propagating terrace 
$((U_{L,k},c_{L,k})_{1\leq k\leq N},(q_{L,k})_{0\leq k\leq N})$ connecting $0$ and $1$, and a unique rightward propagating terrace consisting of a single front $(U_{R},c_{R})$ connecting~$0$ and $1$.
\end{itemize}
\end{theo}
Statement~$(i)$ is a straightforward extension of Theorem~\ref{arb-speed-1D} to the multistable case. Yet the other two statements show some of the much stronger asymmetries which may arise in the shape of propagation of solutions of spatially periodic reaction-diffusion equations. On the one hand, statement $(ii)$ shows that the leftward and rightward terraces may share no platform in common (aside of course from the extremal steady states 0 and 1), regardless of their size. On the other hand, statement $(iii)$ shows that even if the rightward terrace consists in a single front, the leftward terrace may involve an arbitrarily large number of platforms. It is of course possible to mix these two situations to construct even more involved asymmetries.

Regarding the higher dimension case, we also refer to~\cite{gr} where propagating terraces with different shapes in two orthogonal directions were exhibited in dimension 2. Those terraces were found by perturbing a homogeneous reaction term by a periodic function in one direction and keeping the orthogonal direction homogeneous. However, our analysis suggests that such asymmetry may occur even in directions which are not orthogonal, which recalling the Freidlin-Gartner formula~\eqref{eq:fg} above may be more significant in the large-time asymptotics of solutions of the Cauchy problem. This will be addressed in a future work~\cite{giletti-rossi-wip}.

\section{Preliminaries: Properties of pulsating waves}\label{sec:prelim}

In this section, we consider an arbitrary dimension $d \geq 1$ and show some properties of pulsating waves of the spatially periodic bistable equation \eqref{eq:main}. These properties will be used in proving Theorems \ref{arb-speed-1D} and \ref{theo-hd} in later sections.  

We assume that \eqref{eq:main} is of the bistable type in the sense of Definition~\ref{bistable}, where the period vector $L = (L_1, \cdots , L_d)$ is fixed throughout this section. Moreover, we will assume that $f$ is spatially homogeneous on a small neighborhood of $u=0$ and $u=1$, that is, there exists $\delta\in (0,1/2)$ such that 
\begin{equation}\label{homo-neighbor}
f(x,u)=f(u) \,\hbox{ for all }\, x\in\R^d,\, u\in [-\delta,\delta] \cup [1-\delta,1+\delta].
\end{equation}
Notice that these assumptions in particular imply that there exist
$\delta_0 \in (0,\delta)$ and $\mu>0$ such that 
\begin{equation}\label{mu0-delta0} 
f'(u) \leq -\mu  \,  \hbox{ for all }\, u\in [-\delta_0,\delta_0]\cup [1-\delta_0,1+\delta_0].
 \end{equation}
Therefore, the function $u\mapsto f(x,u)$ is decreasing in $[-\delta_0,\delta_0]$ and in $[1-\delta_0,1+\delta_0]$. Then 
by Theorem \ref{th:tw}, for each direction $e \in \mathbb{S}^{d-1}$, there exists a unique speed $c^*(e)$ such that 
\eqref{eq:main} has a pulsating wave $U(t,x)$ with speed $c^*(e)$ in the direction $e$, and $U(t,x)$ is strictly monotone in $t$ provided that $c^*(e)\neq 0$.  

In our discussion below, the direction $e \in \mathbb{S}^{d-1}$ is fixed. Thus by pulsating wave we always mean one which moves in the direction $e$, and for convenience we denote~$c^*:=c^*(e)$.

\subsection{Monotonicity of pulsating wave speeds}

This subsection is concerned with the monotonicity of the wave speed $c^*$ with respect to the nonlinearity $f$. More precisely, we have the following lemma.

\begin{lem}\label{speed-tau-monotone-hd} 
Assume that $f_1$ and $f_2$ are bistable nonlinearities in the sense of Definition~\ref{bistable} with the same period vector~$L$, and that both satisfy \eqref{homo-neighbor}. Let $U_1$ ({\it resp.} $U_2$) be a pulsating wave of \eqref{eq:main} with $f=f_1$ ({\it resp.} $f=f_2$). 

If $f_1 \leq f_2$ in $\R^{d+1}$, then $c_1^* \leq c_2^*$, where $c_1^*$ and $c_2^*$ are the respective speeds of $U_1$ and $U_2$. 
\end{lem}

To prove Lemma \ref{speed-tau-monotone-hd}, we need the following comparison result. 

\begin{lem}\label{sub-super-nonzero}
Let $U(t,x)$ be a pulsating wave of \eqref{eq:main} with speed $c^* \neq 0$ and let $\delta_0>0$, $\mu>0$ be such that \eqref{mu0-delta0} holds. Then there exists $s_0\in \R$, which has the sign of $c^*$, such that if for some $(k_1, \cdots, k_d) \in \mathbb{Z}^d$ and some $\varepsilon \in (0,\delta_0/2]$, 
$$u_0(x)\geq  U(0,x+ kL)-\varepsilon \,\,\hbox{ for } \,\, x\in\R^d,$$
$$({\it resp.} \quad u_0(x)\leq  U(0,x+ kL)+\varepsilon \,\,\hbox{ for } \,\, x\in\R^d),$$
then 
$$ u(t,x;u_0) \geq  U(t-s_0\varepsilon, x+kL)-\varepsilon e^{-\mu t}\,\, \hbox{ for }\,\, t\geq 0,\, x\in\R^d,$$
$$({\it resp.} \quad u(t,x;u_0) \leq  U(t+s_0\varepsilon,x+kL)+\varepsilon e^{-\mu t}\,\, \hbox{ for }\,\,t\geq 0,\, x\in\R^d,)$$
where $u(t,x;u_0)$ is the solution of~\eqref{eq:main} with initial function $u_0$ and $kL = (k_1 L_1, \cdots, k_d L_d)$.
\end{lem}

\begin{proof}
We use the Fife-McLeod type super- and subsolution method to prove Lemma \ref{sub-super-nonzero}.
We only give the construction of a subsolution, as the analysis for a supersolution is identical.  Without loss of generality, we assume that $c^*>0$ (the case where $c^*<0$ can be treated analogously).

Set
$$v(t,x):=U(t-\eta(t),x+kL)-q(t)\ \ \hbox{ for } \ \  t>0,\, x\in\R^d,$$
where $\eta$ and $q$ are $C^1([0,\infty))$ functions such that 
$$\eta(0)=0,\ \  \eta'(t)>0\, \hbox{ for all }\, t\geq 0;  \quad q(0)=\varepsilon,\ \  0<q(t)\leq q(0) \,\hbox{ for all }\, t\geq 0.$$ 
By choosing some appropriate functions $\eta(t)$ and $q(t)$ later, we will show that $v(t,x)$ is a subsolution of \eqref{eq:main}. To do so, for all $(t,x)\in(0,\infty)\times\R^d$,  we define
$$N(t,x):=\partial_tv(t,x)-\Delta v(t,x) -f(x,v(t,x)).$$ 
Since $U(t,x)$ is an entire solution of \eqref{eq:main} and thanks to the spatial periodicity, a straightforward calculation gives that for all $(t,x)\in(0,\infty)\times\R^d$,
$$N(t,x)=-q'(t)-\eta'(t)\partial_tU(t-\eta(t),x+kL)+f(x,U(t-\eta(t),x+kL))-f(x,v(t,x)).$$

Now we choose suitable functions $q(t)$ and $\eta(t)$ such that $N(t,x)\leq 0$ for all $(t,x)\in(0,\infty)\times\R^d$. 
Noticing that $U$ is a pulsating wave with positive speed, we have that $\partial_t U>0$ (see Theorem~\ref{th:tw}). 
Then, on the one hand, for $(\varepsilon,t,x)\in (0,\delta_0/2]\times(0,\infty)\times\R^d$ with $U(t-\eta(t),x+kL))\in [0,\delta_0/2]\cup [1-\delta_0/2,1]$, it follows from \eqref{homo-neighbor} and \eqref{mu0-delta0} that 
$$f(x,U(t-\eta(t),x+kL))-f(x,v(t,x))) =f(U(t-\eta(t),x+kL))-f(v(t,x)))   \leq -\mu q(t),$$
and hence,  
$$ N(t,x)\leq - \eta'(t)\partial_tU(t-\eta(t),x+kL)-q'(t)-\mu q(t) \leq -q'(t)-\mu q(t)$$
(the last inequality follows from the fact that $\partial_tU>0$ and $\eta'(t)>0$ for all $t\geq 0$). 

On the other hand, due to the monotonicity of $U$ in its first variable and the characterization of pulsating waves, one can check that
$$\rho:=\min_{\delta_0/2\le U(t,x)\le1-\delta_0/2}\partial_t U(t,x)>0.$$
It follows that, if $(t,x)\in(0,\infty)\times\R^d$ with $U(t-\eta(t),x+kL)\in [\delta_0/2, 1-\delta_0/2]$, then 
$$N(t,x) \leq -\rho\,\eta'(t)-q'(t)+Cq(t), $$
where $C=\max\{| \partial_uf(x,u)|: x\in [0,L_1]\times\cdots\times[0,L_d], u \in [0,1]\}$.

Let us then choose $q(t)$ and $\eta(t)$ such that
$$ q(0)=\varepsilon,\quad q'(t)=-\mu q(t) \,\,\hbox{ for all }\,\, t\ge0,$$
and 
 $$ \eta(0)=0,\quad \eta'(t)=\frac{C+\mu}{\rho}q(t)\,\, \hbox{ for all }\, \,t\ge0.$$ 
Namely, 
$$q(t)=\varepsilon e^{-\mu t} \, \hbox{ for } \, t\ge0,\quad\hbox{and} \quad  \eta(t)=\frac{\varepsilon(C+\mu)(1-e^{-\mu t})}{\mu\rho} \,\hbox{ for } \,  t\ge0.$$
It is easily checked from the above that $N(t,x)\leq 0$ for all $(t,x)\in (0,\infty)\times \R^d$.

Finally, by the comparison principle, we have that $u(t,x;u_0) \geq v(t,x) $ for  $t\geq 0$, $x\in\R^d$. Taking $s_0=(C+\mu)/(\mu\rho)$ and by the monotonicity of $U$, we easily derive the desired result. The proof of Lemma \ref{sub-super-nonzero} is thus complete.   
\end{proof}

\begin{proof}[Proof of Lemma \ref{speed-tau-monotone-hd}]
Assume by contradiction that $c_1^* > c_2^*$. Then either $c_1^*\neq 0$ or $c_2^*\neq 0$. 

Let us first derive a contradiction in the case where $c_1^*\neq 0$. Let $\delta_0>0$ be such that~\eqref{mu0-delta0} holds. Since $U_1(t,x)$ and $U_2(t,x)$ are pulsating waves connecting $0$ and $1$ in the same direction, for any
$\varepsilon\in (0,\delta_0/2]$, one can find some $(k_1, \cdots, k_d) \in \Z^d$ such that 
$$U_2(0,x) \geq U_1(0,x+kL)-\varepsilon\,\,\hbox{ for }\,\,x\in\R^d,$$
where $k L = (k_1 L_1 ,\cdots ,k_d L_d)$. Clearly, $U_2(t,x)$ is the solution of equation \eqref{eq:main} with $f=f_2$ and initial function $U_2(0,x)$. Let $u(t,x)$ be the solution of \eqref{eq:main} with $f=f_1$ and initial function $U_1(0,x+kL)-\varepsilon$. 
On the one hand, since $f_1\leq f_2$ in $\R^{d+1}$, it follows from the comparison principle that 
$$U_2(t,x) \geq  u(t,x)\,\,\hbox{ for all }\,\, t\geq 0,\, x\in\R^d.$$
On the other hand, by applying Lemma \ref{sub-super-nonzero} to equation \eqref{eq:main} with $f=f_1$, we obtain 
$$u(t,x) \geq  U_1(t-s_0\varepsilon, x+kL)-\varepsilon e^{-\mu t}\,\, \hbox{ for all }\,\, t\geq 0,\, x\in\R^d,$$
where $s_0\in\R$ is a constant which has the sign of $c_1^*$.  Combining the above, we get 
$$U_2(t,x) \geq U_1(t-s_0\varepsilon, x+kL)-\varepsilon e^{-\mu t} \,\,\hbox{ for all }\,\, t\geq 0,\, x\in\R^d.$$ 
We now choose a vector $k^*=(k^*_1,\cdots,k^*_d)\in\Z^d$ such that $k^*L\cdot e / c_1^* >0$, where $k^*L=(k^*_1L_1,\cdots,k^*_dL_d)$. 
Notice that $U_1$ satisfies \eqref{formula-period} with $c=c_1^*$.  
Taking $t=nk^*L\cdot e/c_1^*$ and $x=nk^*L$ for $n\in\N$ in the above inequality yields    
\begin{equation}\label{comp-U1-U2}
U_2\left(\frac{nk^*L\cdot e}{c_1^*},nk^*L\right) \geq U_1(-s_0\varepsilon, kL)-\varepsilon {\rm exp}
\left(\frac{-\mu nk^*L\cdot e}{c_1^*}\right).
\end{equation}
Recall that we have assumed that $c_1^* > c_2^*$. Therefore, 
$$n k^* L \cdot e -  c_2^* \frac{nk^*L\cdot e}{c_1^*} \to +\infty \,\,\hbox{ as }\,\, n\to +\infty, $$
and hence
$$U_2\left(\frac{nk^*L\cdot e}{c_1^*},nk^*L\right) \to 0 \,\,\hbox{ as }\,\, n\to +\infty.$$
Passing to the limit as $n \to +\infty$ in \eqref{comp-U1-U2}, we get $0 \geq U_1 (-s_0 \varepsilon, kL)$, which is a contradiction. 

In the case where $c_2^*\neq 0$, one reaches a similar contradiction by using Lemma \ref{sub-super-nonzero} to construct a supersolution of equation \eqref{eq:main} with $f=f_2$. Therefore, we conclude that  $c^*_1 \leq c^*_2 $, and the proof of Lemma~\ref{speed-tau-monotone-hd} is complete.
\end{proof}

%%%%%%%%%%%%%%%%%%%%%%%%%%%%%%%%%%%%%%%%%%%%%%%%%%%%%%%%%

\subsection{Exponential decay of pulsating waves}

In this subsection, we show the exponential decay of the pulsating waves of \eqref{eq:main} when they approach the stable limiting states $0$ and $1$. 

Recall that $U(t,x)$ is a pulsating wave of \eqref{eq:main} with speed $c^*$. 
By Definition~\ref{pulsating-tw}, we can write 
$$U(t,x):= \Phi(x\cdot e-c^*t,x) \,\, \hbox{ for }  \,\, t\in\R,\,x\in\R^d,$$
provided that $c^*\neq 0$. 
Notice that this change of variables is not available for the pulsating wave when $c^*= 0$, since it is stationary (see also Remark~\ref{rmk:c=0}). When $c^* = 0$, we still use the notation $U$ to refer to a pulsating wave, but we point out here that the profile function is no longer unique in such case (see \cite{dhz2} for a nonuniqueness example in spatial dimension~1).

The following lemma gives the exponential decay of nonstationary pulsating waves. 

\begin{lem}\label{exp-decay-hd-1} 
Let $\lambda_{\pm}$ be positive constants such that 
\begin{equation*}
\lambda_{+}=\frac{c^*+\sqrt{{c^*}^2-4f '(0)}}{2},\qquad  \lambda_{-}=\frac{- c^*+\sqrt{{c^*}^2-4f '(1)}}{2}.
\end{equation*}
If $c^*\neq 0$, then 
\begin{equation}\label{exponen-j-hd}
\frac{\partial_{\xi}\Phi(\xi,x)}{\Phi(\xi,x)}\to -\lambda_{+} \,\hbox{ as } \,\xi\to + \infty,\qquad \frac{\partial_{\xi}\Phi(\xi,x)}{1-\Phi(\xi,x)}=-\lambda_{-}\, \hbox{ as }\, \xi\to-\infty,
\end{equation}
where all the convergences hold uniformly in $x\in\R^d$.
\end{lem}

\begin{proof}
We only show the first convergence stated in  \eqref{exponen-j-hd}, since the proof of the second one is almost identical. We mention that the proof follows the main lines of that of \cite[Proposition 2.2]{h}. For the sake of completeness, we include the details below. 

Notice that $U$ satisfies \eqref{formula-period} with $c=c^*$. By the standard parabolic estimates and the parabolic Harnack inequality, one can conclude that 
\begin{equation}\label{estaime-U-hd}
\sup_{(t,x)\in\R\times\R^d}\left(\frac{|\partial_t U(t,x)|}{U(t,x)} +\frac{| \nabla U (t,x)|}{U(t,x)}\right) \leq C\,\hbox{ for all }  \, (t,x)\in\R\times\R^d ,
\end{equation}
for some constant $C>0$. 
Since $\partial_\xi \Phi (\xi,x)=- \partial_t U ((x\cdot e-\xi)/c^*,x)/c^*$, it then follows that $\partial_\xi \Phi /\Phi$ is globally bounded in $\R\times\R^d$. 
This together with the fact that $\partial_\xi \Phi (\xi,x) < 0$ for $(\xi,x)\in\R\times\R^d$ (due to the monotonicity of $U$ in its first variable, see Theorem \ref{th:tw}) implies that
$$\lambda:=\limsup_{\xi\to + \infty}\sup_{x\in\R^d}\frac{\partial_\xi \Phi (\xi,x)}{\Phi(\xi,x)} \in (-\infty,0].$$
Since $\Phi(\xi,x)$ is $L$-periodic in $x$, one finds a sequence $(\xi_n,x_n)\in \R\times[0,L_1]\times\cdots\times[0,L_d]$ such that $\xi_n\to + \infty$ as $n\to +\infty$ and  
$$\frac{\partial_\xi \Phi (\xi_n,x_n) }{\Phi(\xi_n,x_n) }\to \lambda \,\hbox{ as } \, n\to +\infty.  $$
Up to extraction of some subsequence, we may assume that $x_n\to x_*$ as $n\to +\infty$ for some~$x_*\in [0,L_1]\times\cdots\times[0,L_d]$. 
For each $n\in\N$, set
$$t_n=\frac{x_n\cdot e-\xi_n}{c^*} \quad\hbox{and}\quad V_n(t,x)=\frac{U(t+t_n,x)}{U(t_n,x_n)}\, \, \hbox{ for }\, (t,x)\in\R\times\R^d.$$
Because of \eqref{estaime-U-hd}, the function $V_n$ is locally bounded (that is, $\sup_{n\in\N} \|V_n\|_{L^{\infty}(K)}<\infty$ for any compact subset $K$ of $\R\times\R^d$). Moreover, $V_n$ is positive and satisfies  
\begin{equation*}
\partial_t V_n(t,x)=\Delta V_n(t,x) +\frac{f(x,U(t+t_n,x))}{U(t+t_n,x)}V_n(t,x) \ \,\hbox{ for } \,(t,x)\in\R\times\R^d.
\end{equation*}
By the $C^1$-regularity and the periodicity of $f$, the function $f(x,U(t+t_n,x))/U(t+t_n,x)$ is uniformly bounded. 
Then by standard parabolic estimates, up to extraction of some subsequence, there exists a nonnegative function $V_{\infty}\in C^{1,2}_{loc}(\R\times\R^d)$ such that $V_n \to V_{\infty}$ in $C^{1,2}_{loc}(\R\times\R^d)$. 
Furthermore, since $\xi_n\to +\infty$ and $t_n=(x_n \cdot  e -\xi_n)/c^*$, by Definition~\ref{pulsating-tw}, we have 
$U(t+t_n,x)\to 0$ as $n\to +\infty$ locally uniformly. Then by the assumption~\eqref{homo-neighbor}, 
$$\frac{f(x,U(t+t_n,x))}{U(t+t_n,x)} \to f'(0)\, \hbox{ as }\, n\to +\infty\, \hbox{ locally uniformly for } (t,x)\in\R\times\R^d.$$
Consequently, $V_{\infty}(t,x)$ is a nonnegative solution of 
\begin{equation}\label{eq-Vinfty-hd}
\partial_t V_{\infty}(t,x)=\Delta V_{\infty}(t,x) +f '(0)V_{\infty}(t,x) \, \ \hbox{ for }  \,(t,x)\in\R\times\R^d.
\end{equation}
Notice that $V_{\infty}(0,x_*)=1$. Thus, by the strong maximum principle,  $V_{\infty}(t,x)>0$ for all $(t,x)\in\R\times\R^d$. 

On the other hand, for each $n\in\N$, we have
$$\frac{\partial_t V_n(t,x)}{V_n(t,x)}=\frac{\partial_t U(t+t_n,x)}{U(t+t_n,x)}=-c^*\frac{\partial_{\xi} \Phi(x\cdot e-c^*(t+t_n),x)}{\Phi(x\cdot e-c^*(t+t_n),x)} \ \, \hbox{ for } \, (t,x)\in\R\times\R^d.$$
Combining this with the definition of $\lambda$, we obtain  
$$W(t,x):=\frac{\partial_t V_{\infty}(t,x)}{V_{\infty}(t,x)}\geq -c^*\lambda \,\hbox{ for } \,(t,x)\in\R\times\R^d.$$
By the choice of $(\xi_n,x_n)$ and the definition of $t_n$, we also have $W(0,x_*)=-c^*\lambda$. Furthermore, it is straightforward to check that $W(t,x)$ is a classical solution of the following linear parabolic equation 
$$\partial_t W(t,x)=\Delta W(t,x)+ 2\frac{\nabla V_{\infty}(t,x)}{V_{\infty}(t,x)} \nabla W(t,x) \, \ \hbox{ for }  \,(t,x)\in\R\times\R^d,$$
which reaches its minimum $-c^* \lambda$ at $(0,x_*)$. It then follows from the strong maximum principle that $W \equiv -c^*\lambda$ in $\R\times\R^d$, i.e. that $\partial_tV_{\infty}=-c^* \lambda  V_{\infty}$. Since $V_{\infty}$ satisfies \eqref{eq-Vinfty-hd} and the property \eqref{formula-period} with $c=c^*$, one may check that 
$$V_{\infty}(t,x)=e^{\lambda (x\cdot e-c^*t)}\psi(x)\, \hbox{ for } \,(t,x)\in\R\times\R^d, $$
where $\psi(x)$ is a positive $L$-periodic solution of 
$$\Delta \psi +2\lambda \nabla \psi \cdot  e +(\lambda^2+c\lambda+ f '(0))\psi=0 \, \ \hbox{ for }\, x\in\R^d.$$
Integrating the above equation over $[0,L_1]\times\cdots\times[0,L_d]$, we obtain that $\lambda$ is a root of $\lambda^2+c^*\lambda+ f'(0)=0$. Since $f ' (0) < 0$ and $\lambda$ is nonpositive, it then follows that 
$\lambda=(-c^*-\sqrt{(c^*)^2-4 f '(0)})/2$,   that is, $\lambda=-\lambda_{+}$. 

We have now proved that $\limsup_{\xi\to +\infty}\sup_{x\in\R^d} \partial_\xi \Phi/\Phi=-\lambda_{+}$. 
With a similar argument, one can get that $\liminf_{\xi\to +\infty}\inf_{x\in\R^d}\partial_\xi \Phi /\Phi=-\lambda_{+}$.
Hence, $\lim_{\xi\to +\infty}\partial_\xi  \Phi/\Phi=-\lambda_{+}$. This ends the proof of Lemma \ref{exp-decay-hd-1}.
\end{proof}

\begin{rmk}\label{r-exp-decay-hd}
From Lemma \ref{exp-decay-hd-1}, one can conclude that the nonstationary pulsating waves approach the stable limiting states $0$ and $1$ exponentially fast. For example, for any small constant $0<\epsilon< \lambda_{+}$, one easily checks that there exists $\xi_1>0$ sufficiently large such that 
$$ -(\lambda_{+}+\epsilon)(\xi-\xi_1) \leq \ln \Phi(\xi,x)- \ln\Phi(\xi_1,x) \leq -(\lambda_{+}-\epsilon)(\xi-\xi_1)$$
for all $\xi\geq \xi_1$ and $x\in\R^d$. 
Noticing that $\Phi(\xi,x)$ is periodic in $x$, we can find two positive constants $C_2\geq C_1$ such that
$$C_1 e^{-(\lambda_{+}+\epsilon)\xi} \leq   \Phi(\xi,x) \leq  C_2 e^{-(\lambda_{+}-\epsilon)\xi} \,\,\hbox{ for all }\,\,\xi\geq \xi_1,\,x\in\R^d. $$
Actually, one could further prove that $\lambda_{+}$ is the exponential decay rate of $ \Phi(\xi,x)$ as $\xi\to +\infty$, that is, there exist some positive constant $C$ and some positive $L$-periodic function $\psi\in C^2(\R^d)$ such that
$\Phi(\xi,x)\sim Ce^{ -\lambda_{+}\xi}\psi(x)$ as $\xi\to +\infty$. 
Since this is not needed in showing the main theorems of the present paper, we do not prove it here.
\end{rmk}

Now we turn to the exponential decay of stationary pulsating waves.

\begin{lem}\label{exp-decay-hd-2}
Any stationary (i.e. with $c^* =0$) wave $U$ in direction $e$ satisfies that there exist constants $0<C_{1,\pm}<C_{2,\pm}$ and $\xi_{\pm}>0$ such that 
\begin{equation*}
C_{1,+} e^{-\sqrt{-f'(0)} x\cdot e} \leq  U(x) \leq  C_{2,+} e^{-\sqrt{-f'(0)} x\cdot e} \,\,\hbox{ for all }\,\, x\cdot e \geq \xi_+,
\end{equation*} 
and 
\begin{equation*}
C_{1,-} e^{\sqrt{-f'(1)} x\cdot e} \leq  1 - U(x) \leq  C_{2,-} e^{\sqrt{-f'(1)} x\cdot e} \,\,\hbox{ for all }\,\, x\cdot e \leq  -\xi_-. 
\end{equation*}
\end{lem}

\begin{proof}
We only prove the wanted bounds as $ \xi \to +\infty$, the estimates as $\xi\to-\infty$ being handled by a similar argument.

Let $\delta_0>0$ be a constant such that \eqref{mu0-delta0} holds. 
Recall that $f$ is independent of $x$ when $u\in [-\delta_0,\delta_0]\cup  [1-\delta_0, 1 + \delta_0]$.
We take a function $f_0\in C^1(\R)$ such that 
$$f_0=f\,\,\hbox{ in }\,\, [-\delta_0,\delta_0]\cup  [1-\delta_0, 1 + \delta_0],$$ 
$$f_0<0\,\,\hbox{ in }\,\, \left(0,1/2\right),\quad  f_0>0\,\,\hbox{ in }\,\, \left(1/2,1\right)\,\,\hbox{ and }\,\,\int_0^1f_0(u)du=0.$$
Namely, $f_0$ is a homogeneous balanced bistable nonlinearity with $u=0$ and $u=1$ being two linearly stable zeros. It is well known (see, e.g., \cite{fm1}) that the following equation 
$$\partial_t v =\partial_{\xi\xi} v +f_0(v) \,\,\hbox{ for }\,\, t\in\R,\,\xi\in\R ,$$ 
has a stationary traveling wave $V_0(\xi)$ connecting $0$ and $1$, and that $V_0(\xi)$ approaches $0$ and~$1$ exponentially fast with rates $\sqrt{-f_0'(0)}$ and $\sqrt{-f_0'(1)}$, respectively.

Take some $\xi_+$ such that 
$$U(x) \leq \delta_0\,\,\hbox{ for all }\,\, x \cdot  e \geq \xi_+.$$
In particular, the function $U$ solves
\begin{equation}\label{stationary-U}
\Delta U + f_0 ( U) = 0 \quad \mbox{ for } \ x \cdot  e \geq \xi_+.
\end{equation}
Moreover, it follows from the definition of a stationary pulsating wave (more precisely their asymptotics as $x \cdot e \to \pm \infty$) that for any $\xi \in \R$,
\begin{equation}\label{eq:exp-decay-later}
0 < \inf_{\{x \cdot  e =\xi \}  } U (x)  \leq \sup_{ \{x \cdot  e = \xi \}} U (x) < 1.
\end{equation}
Thus we can find $\xi_1\in\R$ such that
\begin{equation}\label{exp_stat_inf}
\inf_{ \{x \cdot  e = \xi_+ \}} \left( V_0 (x\cdot e + \xi_1) - U (x)  \right)>0.
\end{equation}
Now we claim that
\begin{equation}\label{cla_alternate_exp}
V_0 (x \cdot  e + \xi_1 ) \geq U (x) \ \mbox{ for all } x \cdot  e \geq \xi_+.
\end{equation}
Proceed by contradiction and assume that there exists a sequence $(x_n)_{n \in \mathbb{N}} \subset \R^d$ such that $x_n \cdot  e \geq \xi_+$, and
$$ U (x_n) - V_0 (x_n \cdot  e + \xi_1) \to \sup_{ \{x \cdot  e \geq \xi_+\}} \left(  U (x) - V_0 (x \cdot  e + \xi_1) \right) >0.$$
Since both $U (x)$ and $V_0 (x \cdot  e)$ tend to 0 as $x \cdot  e \to +\infty$, we must have that the sequence $(x_n \cdot  e)_{n \in \mathbb{N}}$ is bounded. Thus up to extraction of a subsequence, we can assume that $x_n \cdot  e \to \xi_\infty \geq \xi_+$ as $n \to +\infty$.

Now we write $x_n = x'_n + x''_n$ with $x'_n = (x_n \cdot  e)  e$ and $x_n'' \perp  e$, and introduce
$$U_n (x) := U (x + x''_n) , \quad V_n(x):= V_0 ( (x + x''_n) \cdot  e + \xi_1). $$
Up to extraction of another subsequence, we have that these two sequences converge respectively to $U_{\infty}$ and $V_{\infty}$ which both solve \eqref{stationary-U}.
Moreover, by construction we have that, for any $x\in\R^d$ such that $x \cdot  e \geq \xi_+$,
$$(U_{\infty}  - V_{\infty} ) (x) \leq (U_{\infty}  - V_{\infty} ) (\xi_\infty  e )  \in (0 , +\infty).$$
Due to \eqref{exp_stat_inf}, we must have $\xi_\infty > \xi_+$. Notice also that, due to our choice of $\xi_+$, we have that $U_{\infty} (\xi_\infty  e) \leq \delta_0$, and hence $V_{\infty} (\xi_\infty  e) \leq \delta_0$. Thus, substracting the equations satisfied by $U_{\infty}$ and $V_{\infty}$ and evaluating at $x=\xi_\infty  e$, we find
$$0 \geq \Delta (U_{\infty}  - V_{\infty} )  (\xi_\infty  e) =  f_0 ( V_{\infty}(\xi_\infty  e) ) - f_0 (U_{\infty}(\xi_\infty  e) ),$$
which is impossible, since by \eqref{mu0-delta0} and by the choice of $f_0$, we have $f_0'<0$ in $[0,\delta_0]$.   
Thus, we have proved~\eqref{cla_alternate_exp} and from the exponential behaviour of $V_0 (\xi)$, we get the wanted upper estimate of $U_0(x)$ as $x\cdot e \to +\infty$. 

The lower estimate as $x\cdot e \to +\infty$ follows from a similar argument. Using~\eqref{eq:exp-decay-later} again, we find $\xi_2\in\R$ such that 
$$\sup_{ \{x \cdot  e = \xi_+ \}} \left( V_0 (x \cdot  e + \xi_2) - U (x)  \right) <0. $$
Then, proceeding exactly as above, we can show that
$$V_0 (x \cdot  e + \xi_2) \leq U (x) \, \mbox{ for all }\, x \cdot  e \geq \xi_+ ,$$
and the wanted estimate follows. We omit the details for the estimates as $x\cdot e\to-\infty$ whose proof is almost identical.
\end{proof}

The above lemma immediately implies that the stationary wave $U(x)$ approaches the limiting states $0$ and $1$ exponentially fast with rates $\sqrt{-f'(0)}$ and $\sqrt{-f'(1)}$, respectively. 
%%%%%%%%%%%%%%%%%%%%%%%%%%%%%%%%%%%%%%%%%%%%%%%%%%

\subsection{Continuity of pulsating wave speeds}

In this subsection, we show the continuity of the wave speed under perturbation on the nonlinearity~$f$.  
Recall that we have assumed that equation \eqref{eq:main} is of the bistable type in the sense of Definition \ref{bistable}, that $f$ satisfies \eqref{homo-neighbor}, and that $U$ is a pulsating wave of \eqref{eq:main} with speed $c^*$ in the direction $e$. To indicate the dependence on $f$, we will write $c^*(f)$ instead of~$c^*$ below.       

The main result of this subsection is stated as follows. 

\begin{lem}\label{speed-tau-continue-hd} 
Let $(f_n)_{n\in\N}\subset C^1(\R^{d+1})$ be a sequence of functions such that each $f_n$ is $L$-periodic in its first variable and is of the bistable type in the sense of Definition \ref{bistable}, and
\begin{equation}\label{fn-converge-f}
 f_n\to f \,\,\hbox{ as }\,\, n\to +\infty \, \, \hbox{ in }  \,\, C^1(\R^{d+1}). 
 \end{equation}
For each $n\in\N$, let $U_n(t,x)$ be a pulsating wave of the equation 
$$\partial_t u  = \Delta u + f_n(x,u) \,\,\hbox{ for } \,\, t \in \R,\, \ x \in \R^d,$$
in the direction $e$ with speed $c^*(f_n)$. Then the following statements hold true:
\begin{itemize}
\item[$(i)$] if $c^*(f)=0$, then $c^*(f_n) \to c^*(f)$ as $n\to +\infty$;
\item[$(ii)$] if $c^*(f)\neq 0$ and there exists some $\sigma>0$ such that $|c^*(f_n)|\geq \sigma$ for all $n\in\N$, then
$c^*(f_n) \to c^*(f)$ as $n\to +\infty$.   
\end{itemize}
\end{lem}

The proof of Lemma \ref{speed-tau-continue-hd} relies on the uniqueness of the wave speed of bistable pulsating waves (see Theorem~\ref{th:tw} above) and the exponential decay of the profiles (see Lemmas~\ref{exp-decay-hd-1} and~\ref{exp-decay-hd-2} above). It shares some similarities with the arguments in~\cite{ag}, where the continuity with respect to the direction was adressed in the ignition case. Notice that in statement $(ii)$, we assume that the sequence of speeds $(c^*(f_n))_{n\in\N}$ is bounded away from $0$ by a positive constant $\sigma$. It should be pointed out that this condition may be relaxed in the sense that $c^*(f_n)\neq 0 $ for all $n\in\N$. Since this general result is not needed in showing our main theorems, we do not include it here.

We also mention that the fact that $c^*(f_n) \to c^*(f)$ as $n\to +\infty$ does not ensure that the pulsating wave $U_n$ converges to $U$ as $n\to +\infty$.  Actually, if $c^*(f)\neq 0$, then one could prove it by our arguments below; yet if $c^*(f)=0$, the convergence may be not true due to the lack of uniqueness of the stationary wave.

The following two lemmas will be used in the proof of Lemma \ref{speed-tau-continue-hd}.

\begin{lem}\label{non-exist-hd} 
Let $u$ be a stationary solution of \eqref{eq:main} and let $\delta_0$ be a positive constant such that \eqref{mu0-delta0} holds. If $0\leq u\leq \delta_0$, then $u\equiv 0$, 
and similarly if $1-\delta_0 \leq u\leq 1$, then $u\equiv 1$. 
\end{lem}

\begin{proof}
We only prove the first assertion, since the second one is similar. Let $u(x)$ be a stationary solution of \eqref{eq:main} such that $0\leq u\leq \delta_0$. Notice that $f(x,u)\equiv f(u)$ for $u\in [0,\delta_0]$. We have $\Delta u+ f(u)=0$ for $x\in\R^d$. On the other hand, let $h(t)$ be the solution of the ODE 
\begin{equation}\label{ode-h}
 \frac{dh}{dt}=-\mu t   \,\,\hbox{ for }\,\, t>0, \quad  h(0)=\delta_0,  
 \end{equation}
where $\mu$ is the positive constant provided by \eqref{mu0-delta0}. Then  a simple comparison argument implies that $u(x) \leq h(t)$ for all $t\geq 0$, $x\in\R^d$.  Since $h(t)\to 0$ as $t\to +\infty$, we immediately obtain $u\equiv 0$. 
\end{proof}

\begin{lem}\label{bound-cn}
The sequence of wave speeds $(c^*(f_n))_{n\in\N}$ is bounded.
\end{lem}
\begin{proof}
Proceed by contradiction and assume without loss of generality that, up to extraction of some subsequence, $0<c^*(f_n)\to+\infty$ as $n\to+\infty$ (as we will sketch below, the case where $0>c^*(f_n)\to-\infty$ as $n\to+\infty$ can be treated similarly). Then by Theorem \ref{th:tw}, for each $n\in\N$, $U_n(t,x)$ is increasing in $t\in\R$, and hence, $U_n(t,\cdot) \to 0$ as $t\to-\infty$ and $U_n(t,\cdot) \to 1$ as $t\to +\infty$ locally uniformly in $\R^d$. By continuity, there exists a unique $t_n\in\R$ such that  
\begin{equation}\label{normal-UR-1-hd}
\max_{x\in[0,L_1]\times\cdots\times[0,L_d]}U_n(t_n,x)=\delta_0, 
\end{equation}
where $\delta_0>0$ is the constant such that \eqref{mu0-delta0} holds. By \eqref{fn-converge-f} and standard parabolic estimates, possibly up to a subsequence, the functions 
$$(t,x)\mapsto U_n(t+t_n,x)  $$
converge in $C_{loc}^{1,2}(\R^{d+1})$ as $n\to +\infty$ to an entire solution $0\leq U_{\infty}(t,x)\leq 1$ of \eqref{eq:main} such that $U_{\infty}(t,x)$ is nondecreasing in $t\in\R$, and $\max_{[0,L_1]\times\cdots\times[0,L_d]}U_{\infty}(0,\cdot)=\delta_0$. Furthermore, since for any $k\in \Z^d$, 
\begin{equation}\label{period-Un}
U_n\left(t+\frac{kL \cdot e}{ c^*(f_n)}, \,x+kL\right) = U_n(t,x) \  \,\hbox{ for }\, t\in\R,\,x\in\R^d,  
\end{equation}
passing to the limit as $n\to +\infty$, we obtain that for any $k\in \Z^d$, 
$$ U_{\infty}\left(t, x+kL\right)=U_{\infty}(t,x)  \ \,\hbox{ for } \,  t\in\R,\,x\in\R^d,$$
that is, $U_{\infty}$ is $L$-periodic in $x$. This implies in particular that $\max_{x\in\R^d} U_{\infty}(0,x)=\delta_0$. 
Then a simple comparison argument implies that $U_{\infty}(t,x) \leq h(t)$ for all $t\geq 0$, $x\in\R^d$, where $h(t)$ is the solution of the ODE \eqref{ode-h}. Thus, $U(t,x) \to 0$ as $t\to +\infty$. This is a contradiction with the fact $U_{\infty}(t,x)$ is nondecreasing in $t\in\R$.  
 
In the case where $0>c^*(f_n)\to-\infty$ as $n\to+\infty$, one derives a similar contradiction by changing the normalization condition  \eqref{normal-UR-1-hd} into 
\begin{equation}\label{normal-UR-2-hd}
\min_{x\in[0,L_1]\times\cdots\times[0,L_d]}U_n(s_n,x)=1-\delta_0
\end{equation}
with $s_n\in\R$. Therefore, we have proved that $(c^*(f_n))_{n\in\N}$ is bounded.    
\end{proof}

We are now ready to prove Lemma \ref{speed-tau-continue-hd}.

\begin{proof}[Proof of Lemma \ref{speed-tau-continue-hd}]
By Lemma \ref{bound-cn}, there exists $c_{\infty} \in\R$ such that, up to extraction of some subsequence, $c^*(f_n) \to  c_\infty$ as  $n\to +\infty$. It suffices to show that 
$$c_\infty =  c^*(f).$$

{\it Proof of statement $(i)$.}  Since $c^*(f)=0$, $U=U(x)$ is a stationary pulsating wave of~\eqref{eq:main}.
Assume by contradiction that $c_\infty \neq 0$. Without loss of generality, we assume that $c_{\infty}> 0$ and $c^*(f_n)>0$ for all large $n\in\N$; then similarly as in the proof of Lemma \ref{bound-cn}, for each large $n\in\N$, there is a unique $t_n$ such that the normalization condition \eqref{normal-UR-1-hd} holds (the case where $c_{\infty}<0$ and $c^*(f_n)<0$ for all large $n\in\N$ can be treated analogously by using the normalization condition \eqref{normal-UR-2-hd}).   
By \eqref{fn-converge-f} and standard parabolic estimates, up to extraction of a subsequence, the functions  $U_n(t+t_n,x)$ 
converge in $C_{loc}^{1,2}(\R^{d+1})$ as $n\to +\infty$ to an entire solution $0\leq U_{\infty}(t,x)\leq 1$ of \eqref{eq:main}. Clearly,  $U_{\infty}(t,x)$ is nondecreasing in $t\in\R$, and $\max_{[0,L_1]\times\cdots\times[0,L_d]}U_{\infty}(0,\cdot)=\delta_0$. Then the strong maximum principle implies that $0<U_{\infty}(t,x)<1$ for all $t\in\R$, $x\in\R^d$. Due to $c_\infty >0$, passing to the limit as $n\to +\infty$ in~\eqref{period-Un} yields that for any $k\in \Z^d$, 
\begin{equation}\label{limit-pulsating-hd}
U_{\infty}\left(t+\frac{kL \cdot  e}{c_\infty},\,x+kL\right)=U_{\infty}(t,x)  \ \,\hbox{ for } \,  t\in\R,\,x\in\R^d.
\end{equation}
Furthermore, by the monotonicity of $U_{\infty}$ in $t$ and standard parabolic estimates, passing to the limit as $t\to\pm\infty$ in \eqref{limit-pulsating-hd}, one finds two $L$-periodic steady states $u_{\pm}(x)$ of \eqref{eq:main} such that 
\begin{equation}\label{limit-tilde-u-hd}
 U_{\infty}(t,x) \to  u_{\pm}(x)  \,\hbox{ as } \,  t\to\pm\infty\,  \hbox{ locally uniformly for }\, x\in\R^d.  
 \end{equation}
It is easily seen that $0\leq u_{-}(x) \leq  u_{+}(x) \leq 1$ for $x\in\R^d$, and that 
$$\max_{[0,L_1]\times\cdots\times[0,L_d]}u_{-}(\cdot)  \leq \max_{[0,L_1]\times\cdots\times[0,L_d]}U_{\infty}(0,\cdot)=\delta_0 \leq \max_{[0,L_1]\times\cdots\times[0,L_d]}u_{+}(\cdot).$$ 
Then Lemma \ref{non-exist-hd} immediately implies that $u_{-}\equiv 0$. 

Moreover, by the strong maximum principle, either $u_{+}\equiv 1$ or $0<u_{+}< 1$. We will derive a contradiction in each of these two cases. The following change of variable will be useful in the proof below:
\begin{equation}\label{Uinfty-form}
\Psi (\xi, x):=U_\infty \left(\frac{x\cdot e-\xi}{c_\infty},x\right) \,\,\hbox{ for }\,\, (\xi,x)\in\R\times\R^d. 
\end{equation}

If $u_{+}\equiv 1$, then in view of \eqref{Uinfty-form}, one can check that, in the direction $ e$, $U_{\infty}(t,x)$ is a pulsating wave of \eqref{eq:main} connecting $0$ and $1$ in the sense of Definition~\ref{pulsating-tw}, and $c_\infty>0$ is its wave speed. Yet, remember that $U$ is a stationary wave of~\eqref{eq:main} connecting $0$ and $1$ in the same direction. This is impossible, due to the uniqueness of the wave speed of pulsating waves in a given direction (see Theorem~\ref{th:tw}).

Next, we find a contradiction in the case where $0<u_{+}< 1$, that is, $U_{\infty}(t,x)$ is a pulsating wave of \eqref{eq:main} connecting $0$ and $u_{+}$ with positive speed $c_\infty$ in the direction $e$.  

Let us first show that 
there exists some $\xi_*\in [0,\infty)$ such that 
\begin{equation}\label{Utau-Phi}
U(x) \geq  \Psi(x\cdot e+\xi_*,x) \,\,\hbox{ for all }\,\, x\in\R^d,   
\end{equation}
where $\Psi$ is the function defined in \eqref{Uinfty-form}.
Since $U$ is a stationary wave in the direction~$ e$, it follows from Lemma~\ref{exp-decay-hd-2} that there exist some $M_1>0$ and $C_1>0$ such that 
$$U(x) \geq  C_1 e^{-\lambda x\cdot e} \,\,\hbox{ for all }\,\, x\cdot e\geq M_1,$$ 
where $\lambda=\sqrt{-f '(0)}$. On the other hand, proceeding as in the proof of Lemma \ref{exp-decay-hd-1}, one can conclude that 
\begin{equation*}
\frac{\partial_\xi \Psi (\xi,x)}{\Psi(\xi,x)} \to -\tilde{\lambda}\, \hbox{ as }\, \xi\to + \infty \,\hbox{ uniformly in }\, x\in\R^d,
\end{equation*}
where $\tilde{\lambda}=(c_\infty +\sqrt{c_\infty^2-4 f '(0)})/2$. Since $c_\infty>0$, it is easily checked that $\tilde{\lambda} > \lambda$. Let $\epsilon>0$ be a small constant such that $\lambda <\tilde{\lambda}-\epsilon$. Then, from the discussion in Remark~\ref{r-exp-decay-hd}, one finds some $M_2>0$ and $C_2>0$ such that
$$ \Psi(x\cdot e,x) \leq  C_2 e^{-(\tilde{\lambda}-\epsilon)x\cdot e} \,\,\hbox{ for all }\,\, x\cdot e\geq M_2.$$
This means that $\Psi(x\cdot e,x)$ decays faster than $U(x)$ as $x\cdot e\to +\infty$. Thus, there exists some $M_3>0$ such that 
$$U(x) \geq   \Psi(x\cdot e,x)  \,\, \hbox{ for all }\,\, x\cdot e\geq M_3.$$

Now, suppose by contradiction that \eqref{Utau-Phi} is not true for any $\xi_* \geq 0$. Then in view of the above property and the fact that $\Psi$ is decreasing in its first variable, one can find sequences $(\xi_n)_{n\in\N}\subset [0,\infty)$ and $(x_n)_{n\in\N}\subset\R^d$ such that $\xi_n\to +\infty$ as $n\to +\infty$, and that for each $n\in\N$, $x_n \cdot e < M_3$ and 
\begin{equation*}
U(x_n) <  \Psi(x_n\cdot e+\xi_n,x_n).   
\end{equation*}
However, this is impossible. On the one hand, if the sequence $(x_n \cdot e)_{n\in\N}$ is bounded, then we have $\liminf_{n\to +\infty}U(x_n)>0$ (this follows from the definition of a stationary wave, recall also~\eqref{eq:exp-decay-later}), while $\Psi(x_n\cdot e+\xi_n,x_n)\to 0$ as $n\to +\infty$. On the other hand, if $(x_n \cdot e)_{n\in\N}$ is unbounded, then $x_n \cdot e\to-\infty$ as $n\to +\infty$, whence
$U(x_n) \to 1$ as $n\to +\infty$, while 
$0< \Psi(x_n\cdot e+\xi_n,x_n)<u_+(x_n)$ for all $n\in\N$ (recall that $u_+$ is a periodic function strictly between $0$ and $1$). Therefore, we can conclude that there exists some $\xi_*\in [0,\infty)$ such that~\eqref{Utau-Phi} holds.

In view of \eqref{Uinfty-form}, \eqref{Utau-Phi} implies that 
$$ U(x) \geq  U_{\infty}(-\xi_*/c_{\infty},x)  \,\,\hbox{ for all }\, \,x\in\R^d.$$ 
Then, applying the comparison principle to equation \eqref{eq:main}, we obtain 
$$U(x) \geq U_{\infty}(t-\xi_*/c_{\infty},x)  \,\,\hbox{ for all }\, \, t\geq 0,\,x\in\R^d.$$ 
Passing to the limit as $t \to +\infty$, we get that $U\geq u_+$ in $\R^d$, which clearly contradicts the fact that $U$ is a stationary wave connecting 1 and 0. This means that our assumption $c_{\infty}>0$ at the beginning was false. Therefore, we have proved that $c_\infty=c^*(f)=0$. \smallskip

{\it Proof of statement $(ii)$.} Without loss of generality, we may assume that $c^*(f)>0$ (the case where $c^*(f)<0$ can be treated similarly). In this case, $U$ is a pulsating wave of \eqref{eq:main} connecting $0$ and $1$ with positive speed. 
By the sign property of the wave speed (see \eqref{sign-property}), it follows that 
$$\int_0^1\int_{(0,L_1)\times \cdots\times(0,L_d)} f(x,u)dxdu >0.$$ 
Due to \eqref{fn-converge-f}, we have
$$\int_0^1\int_{(0,L_1)\times \cdots\times(0,L_d)} f_n(x,u)dxdu >0\,\,\hbox{ for all  large }\, n\in\N.$$ 
This together with the assumption $|c^*(f_n)|> \sigma$ for all $n\in\N$ and the sign property of the wave speed implies that $c^*(f_n)$ must be positive for all $n\in\N$, and hence, $c_\infty \geq  \sigma$.

Assume by contradiction that $c_\infty \neq c^*(f)$. Then either $c_\infty> c^*(f)>0$ or $0<c_\infty< c^*(f)$.   
Let us first derive a contradiction in the former case. The argument is actually quite similar to that in the proof of statement $(i)$; therefore we only give its outline.

Since $c^*(f_n)>0$ for all $n\in\N$, one can find a sequence $(t_n)_{n\in\N}\subset \R$ satisfying~\eqref{normal-UR-1-hd}. Then, up to extraction of some subsequence, the functions $(t,x)\mapsto U_n(t+t_n,x) $ converge in $C_{loc}^{1,2}(\R\times\R^d)$ to an entire solution $0< U_{\infty}(t,x)< 1$ of~\eqref{eq:main} such that $U_{\infty}(t,x)$ is nondecreasing in $t\in\R$, satisfies~\eqref{eq:main}, and $\max_{[0,L_1]\times\cdots\times[0,L_d]}U_{\infty}(0,\cdot)=\delta_0$.

Furthermore, $U_{\infty}$ connects two periodic steady states $u_{\pm}(x)$ as in \eqref{limit-tilde-u-hd}. Then $u_{-}\equiv 0$, and either $u_{+}\equiv 1$ or $0<u_{+}< 1$. Namely, either $U_{\infty}(t,x)$ is a pulsating wave of~\eqref{eq:main} connecting $0$ and $1$ with speed $c_{\infty}$ or it is a pulsating wave connecting~$0$ and~$u_{+}$ with speed~$c_{\infty}$. The former is impossible, since it contradicts the uniqueness of the speed of bistable pulsating waves in a given direction. The latter also leads to a contradiction, since by similar arguments to those used in the proof of statement $(i)$ (notice from Lemma \ref{exp-decay-hd-1} that the decay of a pulsating wave going to 0 becomes faster when the speed increases), one can find some $\xi_*\in [0,\infty)$ such that 
 $$U_{\infty}(t-\xi_*/c_{\infty},x)\leq  U(t,x) \ \,\hbox{ for }\, t\geq 0,\, x\in\R^d .$$
This indeed contradicts our assumption that $ c_\infty >c^*(f)$. 

If $c_\infty < c^*(f)$, one reaches a similar contradiction by changing the normalization condition  \eqref{normal-UR-1-hd} into  \eqref{normal-UR-2-hd}. Thus, we have proved that $c_\infty =c^*(f)$ if $c^*(f)>0$. In the case where $c^*(f)<0$, the proof is almost identical; therefore we omit the details. 
\end{proof}

%%%%%%%%%%%%%%%%%%%%%%%%%%%%%%%%%%%%%%%%%%%%%%%%%%%%%%%%%

\section{Admissible speeds in the one-dimensional case}\label{sec:1d}

In this section, we fix $d=1$ and prove Theorem \ref{arb-speed-1D}. The proof consists of two steps. First we 
construct a spatially periodic nonlinearity $f$ such that the corresponding equation is bistable in the sense of Definition \ref{bistable} and admits a pulsating wave with positive speed to the left, but a  pulsating wave with zero speed to the right. Secondly, by perturbing the above nonlinearity $f$ and using a rescaling argument, we prove that any pair of two speeds $c_L $, $c_R $ with $c_L c_R\geq 0$ is admissible, respectively, in the leftward and rightward directions.

\subsection{Simultaneous zero and positive speeds in opposite directions}

This subsection is devoted to the first step of the proof of Theorem~\ref{arb-speed-1D}. We first introduce $f_0$ of the balanced Allen-Cahn type
\begin{equation}\label{eq:balanced}
f_0 (u) = u(1-u)(u-1/2).
\end{equation}
It is well known (see, e.g., \cite{aw,fm1}) that the following equation 
\begin{equation}\label{Allen-Cahn}
\partial_t u=\partial_{xx}u+ f_0(u), \quad  t\in\R,\,x\in\R ,
\end{equation}
has a stationary traveling wave $u(t,x)=U_0(x)$ decreasing in $x$ and connecting $0$ and $1$ in the rightward direction, that is, $\lim_{x\to -\infty}U_0(x)=1$ and $\lim_{x\to -\infty}U_0(x)=0$, and that $U_0$ is the unique (up to shifts) rightward traveling wave.
Since equation \eqref{Allen-Cahn} is invariant under the spatial reflection $x\to -x$, it is easily seen that $U_0(-x)$ is the unique (up to shifts) leftward traveling wave connecting $0$ and $1$. We normalize $U_0$ so that  $U_0 (0) = 1/2$. 

Let $E$ be a subset of $\R \times [0,1]$ defined as follows:
\begin{equation*}%\label{define-E}
E=\left\{(x,u):\, U_0 \left(x\right) \leq u < U_0 \left(x-1\right) \right\}. 
\end{equation*}
Since $U_0(x)$ is decreasing in $x\in\R$, it is clear that for any $(x,u)\in \R\times (0,1)$, there exists a unique $m\in\Z$ such that $(x-m,u)\in E$.

Let $\delta_0 \in (0,1/4)$ be a constant such that 
\begin{equation}\label{choose-delta0}
f_0'(u) \leq \frac{\max\{f_0'(0),\,f_0'(1)\}}{2} < 0\,\,\hbox{ for all }\,\, u\in [-\delta_0,\delta_0]\cup [1-\delta_0,1+\delta_0]. 
\end{equation}
Now we take a smooth function $\chi: E \to [0,1]$ satisfying
$$ \chi(x,u)=0 \,\hbox{ for } \,x \in\R,\,u\in [0,\delta_0]\cup[1-\delta_0,1], $$
$$\chi=0 \,\hbox{ on }  \partial E,$$
and
$$\chi>0 \,\hbox{ in }\, \left\{(x,u)\in E \backslash \partial E: 2\delta_0\leq u\leq 1-2\delta_0 \right\}.$$
Next, we extend $\chi$ to any $(x,u)\in \R\times [0,1]$ in the following way
$$ \chi(x,u):= \chi(x-m,u) $$
where $m$ is the unique integer such that $(x-m,u)\in E$. Without loss of generality, we can also assume that 
$$\partial_{x} \chi(x,u)=0\,\,\hbox{ if } \,\, (x-m,u)\in \partial E$$
and that 
$$ \chi(x,u)=0 \,\,\hbox{ for all } \,\,x \in\R,\,u\in (-\infty,0)\cup (1,+\infty).$$
Then the resulting function 
$$\chi: \R\times \R \to [0,1] $$ 
is of class $C^1$ and it is $1$-periodic with respect to $x$.

We then define 
$$f_{\sigma}(x,u) = f_0 (u)+ \sigma\chi (x,u)\,\, \hbox{ for }\,\, x\in\R,\, u\in\R,$$
where $\sigma$ is a positive constant to be determined later. We now consider the corresponding spatially periodic equation
\begin{equation}\label{eq-sigma}
\partial_t u = \partial_{xx} u + f_{\sigma}(x,u)\,\, \hbox{ for }\,\, t\in\R,\, x\in\R,
\end{equation}
and show the following:
\begin{prop}\label{prop-zero-pos}
There exists $\sigma_*>0$ such that for any $0<\sigma<\sigma_*$, the following statements hold true:
\begin{itemize}
\item[$(i)$] equation~\eqref{eq-sigma} is of the spatially periodic bistable type in the sense of Definition~\ref{bistable};
\item[$(ii)$] equation~\eqref{eq-sigma} has a leftward pulsating wave with speed $c^*_L >0$ and a rightward pulsating wave with speed $c^*_R = 0$.
\end{itemize}
\end{prop}

The key point of the proof of Proposition~\ref{prop-zero-pos} is that  $U_0(x)$ is a stationary front of \eqref{eq-sigma}, which blocks the propagation in the right direction, while $U_0 (-x + x_0)$ is a strict subsolution for any shift $x_0 \in \R$ (this is due to the fact that, by our choice of $\chi$, any shift of $U_0 (-x)$ has to intersect the support of $\chi$), which forces propagation with positive speed in the left direction. It should be pointed out that the above properties may hold under more general conditions on the nonlinearity $f$ and our construction is in no way unique. In particular, we mention that the possibility of blocking propagation in only one of two opposite directions was also explored in the context of periodic domains with holes~\cite{dr}.

We now give the proof of Proposition \ref{prop-zero-pos}. We start by checking that \eqref{eq-sigma} is of the bistable type with period $L=1$. It is straightforward to check that the steady states $u\equiv 0$ and $u\equiv 1$ are linearly stable, since $\chi$ is null on a neighborhood of $u = 0$ and $u=1$. Then, the Dancer-Hess connecting orbit theorem implies that there exists at least one $1$-periodic steady state~$\bar{u}$ such that $0<\bar{u}<1$ in $\R$. The following lemma shows that any such $\bar{u}$ is linearly unstable provided that $\sigma$ is small enough.

\begin{lem}\label{inter-unstable}
There exists $\sigma_* >0$ such that, for any $0 < \sigma <\sigma_*$ and any $\bar{u}$ a $1$-periodic steady state of \eqref{eq-sigma} with $0<\bar{u}<1$, then $\lambda_1(\sigma,\bar{u})>0 $ where $\lambda_1 (\sigma, \bar{u})$ denotes the principal eigenvalue of 
\begin{equation*}%\label{p-eigen}
\psi''+\partial_uf_{\sigma}(x,\bar{u}(x))\psi=\lambda\psi\hbox{ in }\R,\quad \psi>0\hbox{ in }\R,\quad \psi\hbox{ is }1\hbox{-periodic}.
\end{equation*}
\end{lem}  
For later use (see Lemma \ref{exist-pf-hd}), we prove Lemma~\ref{inter-unstable} by arguments that also apply in the multi-dimensional case.
 \begin{proof}[Proof of Lemma \ref{inter-unstable}]
 Assume by contradiction that there are some sequences $(\sigma_n)_{n\in\N}$ in $(0, \infty)$, $(\bar{u}_n)_{n\in\N}$ and~$(\psi_n)_{n\in\N}$ in $C^2(\R)$ such that $\sigma_n\to 0$ as $n\to +\infty$ and, for each $n\in\N$, the functions $\bar{u}_n$ and $\psi_n$ respectively satisfy
\begin{equation*}
\bar{u}_n''+f_{\sigma_n}(x,\bar{u}_n)=0\ \hbox{ in }\R,\ \ \bar{u}_n\hbox{ is }1 \hbox{-periodic, }\ 0<\bar{u}_n <1\hbox{ in }\R,
\end{equation*}
\begin{equation*}
\psi_n''+\partial_uf_{\sigma_n}(x,\bar{u}_n)\psi_n=\lambda_1(\sigma_n,\bar{u}_n)\psi_n\ \hbox{ in }\R,\ \ \psi_n\hbox{ is }1 \hbox{-periodic},\ \ \psi_n>0\hbox{ in }\R,
\end{equation*}
with $\lambda_1(\sigma_n,\bar{u}_n)\leq 0$. Since
\begin{equation*}
\min_{x\in\R,\,u\in [0,1]} \partial_uf_{\sigma_n}(x,u) \leq \lambda_1({\sigma_n},\bar{u}_n) \leq \max_{x\in\R,\,u\in [0,1]}\partial_uf_{\sigma_n}(x,u),
\end{equation*}
the sequence $\big(\lambda_1(\sigma_n,\bar{u}_n)\big)_{n\in\N}$ is then bounded. Up to extraction of some subsequence, there is a real number $\tilde{\lambda}_1\leq 0$ such that $\lambda_1(\sigma_n,\bar{u}_n)\to\tilde{\lambda}_1$ as $n\to +\infty$. 

By standard elliptic estimates, there is a $C^2(\R)$ function $ u_{\infty}$ such that, up to extraction of some subsequence, $\bar{u}_n\to u_{\infty}$ in $C^2(\R)$ as $n\to +\infty$, and the function $u_{\infty}$ satisfies 
$$u_{\infty}''+ f_0 (u_{\infty})=0 \ \hbox{ in }\R,\ \ u_{\infty}\hbox{ is }1 \hbox{-periodic, }\ 0\leq u_{\infty} \leq 1\hbox{ in }\R.$$ 
Similarly, by normalizing $\psi_n$ in such a way that $\max_{x\in\R} \psi_n(x)=1$, one finds a nonnegative $C^2(\R)$ function $\psi_{\infty}$ such that, possibly up to a further subsequence, $\psi_n\to\psi_{\infty}$ in $C^2(\R)$ as $n\to +\infty$, and the function $\psi_{\infty}$ satisfies 
\begin{equation}\label{eq-psi-infty}
\psi_{\infty}''+ f_0 ' (u_{\infty})\psi_{\infty}=\tilde{\lambda}_1\psi_\infty\ \hbox{ in }\R,\ \ \psi_{\infty}\hbox{ is }1 \hbox{-periodic},\ \  \max_{x\in\R}\psi_{\infty}(x) =1.
\end{equation}
By the strong maximum principle, we see that the function $\psi_\infty$ is also positive.

Next, we observe that $u_{\infty}  \not\equiv  0$ and $u_{\infty}  \not\equiv  1$. 
Otherwise, for sufficiently large $n$, there would hold $\bar{u}_n \in (0,\delta_0] \cup [1-\delta_0,1)$,
where $\delta_0\in (0,1/4)$ is the constant given in \eqref{choose-delta0}. Since each $\chi$ is a null on $(0,\delta_0] \cup [1-\delta_0,1)$, by the definition of $f_{\sigma}$, $\bar{u}_n$ would satisfy 
$$\bar{u}''_{n}+ f_0 (\bar{u}_{n})=0 \ \hbox{ in }\R,\ \ \bar{u}_{n}\hbox{ is }1 \hbox{-periodic.}$$ 
Integrating the above equation over $[0,1]$, one obtains 
$\int_0^1 f_0(\bar{u}_n(x))dx=0$,
which is a contradiction with the fact that $f_0<0$ on $(0,\delta_0]$ and $f_0>0$ on $[1-\delta_0,1)$.
Therefore, $u_{\infty} \not\equiv 0$ and $u_{\infty} \not\equiv 1$.

As a consequence, two cases may happen: either $u_{\infty}$ is a constant strictly between $0$ and $1$, i.e. $u_{\infty}\equiv \frac{1}{2}$, or $u_{\infty}$ is a non-constant periodic solution. However, any such steady state of a homogeneous bistable equation is linearly unstable, which contradicts the fact that $\tilde{\lambda}_1 \geq 0$. For the sake of completeness, we provide the details below.

On the one hand, if $u_{\infty}\equiv \frac{1}{2}$, then we have $\tilde{\lambda}_1= f_0 '(u_{\infty})>0$, which indeed contradicts the assumption that $\tilde{\lambda}_1\leq 0$. On the other hand, if~$u_{\infty}$ is a non-constant periodic solution, then $u'_{\infty}$ is a sign-changing periodic solution of 
$$(u'_{\infty})''+  f_0 ' (u_{\infty})u'_{\infty} =0 \ \hbox{ in }\R.$$
This means that $0$ is an eigenvalue of the linear operator $\mathcal{L}\varphi:= \varphi'' +f_0 ' (u_{\infty})\varphi$ in the space of periodic functions. Moreover, it follows from \eqref{eq-psi-infty} that $\tilde{\lambda}_1$ is an eigenvalue of the operator with positive eigenfunction. By the Krein-Rutmann theory, $\tilde{\lambda}_1$ is the principal eigenvalue which is maximal and simple. This implies that $\tilde{\lambda}_1>0$, which is again a contradiction. The proof of Lemma \ref{inter-unstable} is thus complete.
 \end{proof}

Let $\sigma_*$ be the positive constant provided by the above lemma. By Theorem \ref{th:tw}, for every $0<\sigma<\sigma_*$, equation \eqref{eq-sigma} has a leftward pulsating wave $U_L$ with speed $c^*_L$ and rightward pulsating wave $U_R$ with speed $c^*_R$. Furthermore, since 
$$\int_{0}^1\int_{0}^1 f_{\sigma}(x,u)dxdu>0,$$
it follows from the sign property of wave speeds (see \eqref{sign-property}) that both $c^*_L$ and $c^*_R$ are nonnegative. The next lemma deals with the signs of these two speeds and completes the proof of Proposition \ref{prop-zero-pos}.

\begin{lem}\label{sign-rightleft}
The rightward pulsating wave speed $c^*_R$ must be zero, i.e. there is no non-stationary pulsating wave in the right direction, and the leftward pulsating wave $U_L$ has speed $c^*_L >0$.
\end{lem}

\begin{proof}
By our construction of $f_{\sigma}$, it is straightforward to check that $U_0(x)$ is a rightward stationary wave of \eqref{eq-sigma}. Then $c^*_R = 0$ from the uniqueness of speeds of bistable pulsating waves (see Theorem \ref{th:tw}).

Next, suppose by contradiction that $c^*_L$ is not positive. Then $c^*_L=0$, that is, there exists $U_L=U_L(x)$ a stationary wave. We will derive a contradiction by several steps.\smallskip

{\it Step 1:  Construction of a subsolution.} Let $\mu=- \max \{f_0'(0),\,f_0'(1)\}/2$ and $\delta_0>0$ be a positive constant such that \eqref{choose-delta0} holds.  Then we have 
\begin{equation}\label{eq-mu} 
f_0 '(u) \leq -\mu  \ \  \hbox{ for all }\ \   u\in [-\delta_0,\delta_0]\cup [1-\delta_0,1+\delta_0].
\end{equation}
In this step, we show that there exists a positive constant $K$ such that if for some $\varepsilon \in (0,\delta_0/2]$ and $x_0\in\R$, 
$$u_0(x)\geq  U_0(-x+x_0)-\varepsilon \ \ \hbox{ for } \ \ x\in\R, $$
then $u(t,x;u_0)$ the corresponding solution of~\eqref{eq-sigma} satisfies
\begin{equation*} 
u(t,x;u_0) \geq  U_0(-x+x_0+K\varepsilon)-\varepsilon e^{-\mu t}\ \ \hbox{ for }\ \ t\geq 0,\, \,x\in\R.  
\end{equation*}
 The proof follows from similar arguments to those used in the proof of Lemma \ref{sub-super-nonzero}. For completeness and also for convenience of later applications, we include the details below.   
 
Set
$$v(t,x):=U_0(-x+\eta(t))-q(t)\ \ \hbox{ for } \ \  t>0,\, \,x\in\R,$$
where $\eta$ and $q$ are $C^1([0,\infty))$ functions  to be determined later such that 
$$\eta(0)=x_0,\ \  \eta'(t)>0 \,\hbox{ for all }\, t\geq 0;  \quad q(0)=\varepsilon,\ \  0<q(t)\leq q(0) \,\hbox{ for all }\, t\geq 0.$$ 
For all $(t,x)\in(0,\infty)\times\R$, define
$$N(t,x):=\partial_t v(t,x)-\partial_{xx}v(t,x) -f_{\sigma}(x,v(t,x)).$$ 
Since $U_0$ is a solution of \eqref{Allen-Cahn}, a straightforward calculation gives that for $(t,x)\in(0,\infty)\times\R$,
$$ N(t,x)=-q'(t)+\eta'(t)U_0'(-x+\eta(t))+ f_0 (U_0(-x+\eta(t)))-f_{\sigma}(x,v(t,x)).$$
Notice that $U'_0(x)<0$ for $x\in\R$. Then, for $(\varepsilon,t,x)\in (0,\delta_0/2]\times(0,\infty)\times\R$ with $U_0(-x+\eta(t))\in [ 0,\delta_0/2]\cup [1-\delta_0/2,1]$, it follows from the definition of $f_{\sigma}$ and \eqref{eq-mu} that 
$$ N(t,x)\leq -q'(t)+\eta'(t)U_0'(-x+\eta(t))-\mu q(t)\leq  -q'(t)-\mu q(t).$$
On the other hand, due to the monotonicity of $U_0$, there is a constant 
$$\rho:=-\max_{\delta_0/2\le U_0(y)\le1-\delta_0/2}U_0'(y)>0,$$  
such that if $(t,x)\in(0,\infty)\times\R$ with $U_0(-x+\eta(t))\in [\delta_0/2, 1-\delta_0/2]$, then 
\begin{equation*}
\left.\begin{array}{ll} 
\medskip N(t,x)\!\!\!& \leq -\rho\,\eta'(t)-q'(t)+f_0 (U_0 (-x+\eta(t)))-f_{\sigma}(x,v(t,x))  \\ 
\medskip & \leq -\rho\,\eta'(t)-q'(t)+f_0 (U_0 (-x+\eta(t)))-f_0 (U_0 (-x+\eta(t))-q(t))  \\ 
& \leq  -\rho\,\eta'(t)-q'(t)+Mq(t), 
\end{array}\right. 
\end{equation*}
where $M=\max_{u \in [0,1]}| f_0 ' (u)|$, and the second inequality follows from the fact that $\chi$ is nonnegative.

Let us then choose $q(t)$ and $\eta(t)$ such that
\begin{equation*}
q(0)=\varepsilon,\quad q'(t)=-\mu q(t)\,\hbox{ for all } \, t\ge0,
\end{equation*}
and
$$ \eta(0)=x_0,\quad \eta'(t)=\frac{M+\mu}{\rho}q(t)\, \hbox{ for all } \, t\ge0. $$
It is then clear that $N(t,x)\leq 0$ for all $(t,x)\in (0,\infty)\times \R$. 
Taking a positive constant $K$ such that $K\geq (M+\mu)/(\mu\rho)$, 
one can easily derive the desired result by using the comparison principle. This ends the proof of {\it Step 1}. \smallskip

{\it Step 2: Comparison of $U_L(\cdot)$ and a shift of $U_0(-\cdot)$.} It is clear that for any  $\varepsilon\in (0,\delta_0/2]$, there exists some $x_0\in\R$ such that $U_L(x) \geq U_0(-x+x_0)-\varepsilon$ for $x\in\R$.   
Then by {\it Step 1}, we obtain
$$U_L(x) \geq U_0(-x+x_0+K\varepsilon)-\varepsilon e^{-\mu t}\, \, \hbox{ for }\, \, t\geq 0,\,\, x\in\R. $$
Passing to the limit as $t\to +\infty$, we have $U_L(x) \geq U_0(-x+x_0+K\varepsilon)$ for $x\in\R$.  Now we can define 
\begin{equation}\label{define-xi*}
\xi_*:=\inf\{\xi\in\R:\, U_L(x) \geq U_0(-x+\xi)  \,\hbox{ for } \,x\in\R  \}. 
\end{equation}
Clearly, $\xi_*$ is a real number and $U_L(x) \geq U_0(-x+\xi_*)$ for $x\in\R$. 

Next, we show that 
\begin{equation}\label{UL-U0}
U_L(x)> U_0(-x+\xi_*)\,\hbox{ for all }\, x\in\R. 
\end{equation}
Notice that $U_0 (-x + \xi_*)$ is a subsolution of \eqref{eq-sigma}, but it is not a solution, since it intersects the support of $\chi$ where $f_{\sigma}(x,u) > f_0 (u)$. It then follows from the strong maximum principle that the solution $u(t,x)$ of \eqref{eq-sigma} starting from $u(0,x):=U_0 (-x + \xi_*) $ is increasing in $t\in\R$. This in particular implies that 
$$u(t,x) > U_0 (-x + \xi_*)\,\,\hbox{ for all }\,\, t>0,\,x\in\R.$$ 
On the other hand, since $U_L(x)$ is a stationary solution, the comparison principle implies
$$u(t,x) \leq  U_L(x) \,\,\hbox{ for all } \,\,t>0,\,x\in\R.    $$
Combining the above, we immediately obtain \eqref{UL-U0}. \smallskip

{\it Step 3: Completion of the proof.} Take a large positive constant $C$ such that 
$$\hat{\delta}:= \sup_{|x|\geq C-1} | U'_0(-x+\xi_*)  | \leq \frac{1}{2K},$$
where $K$ is the constant obtained in {\it Step 1}. Without loss of generality, we may assume that 
$K\geq 1$, and hence we have $\hat{\delta}\leq 1/2$ here.  
Because of \eqref{UL-U0}, we can find some small constant $\hat{\varepsilon}\in (0,\delta_0/2)$ such that 
$$U_L(x) > U_0(-x+\xi_*-\hat{\varepsilon}) \ \ \hbox{ for }\ \  x\in [-C,  C].$$
On the other hand, for all $x\in (-\infty,-C)\cup (C,\infty)$, by the definition of $\hat{\delta}$,  we have 
\begin{equation*}
\left.\begin{array}{ll}
\medskip U_L(x) - U_0(-x+\xi_*-\hat{\varepsilon}) \!\!\!\!& =\, U_L(x)-U_0(-x+\xi_*)+U_0(-x+\xi_*)  -U_0(-x+\xi_*-\hat{\varepsilon}) \\
 & \geq   -\max_{|x|\geq C}|U_0(-x+\xi_*)  -U_0(-x+\xi_*-\hat{\varepsilon})| \geq -\hat{\varepsilon}\hat{\delta}. 
\end{array}\right. 
\end{equation*}
 Combining the above, we obtain 
$$U_L(x) \geq  U_0(-x+\xi_*-\hat{\varepsilon})-\hat{\varepsilon}\hat{\delta}  \ \ \hbox{ for all }\ \  x\in \R.$$
Notice that $\hat{\varepsilon} \hat{\delta}\in (0,\delta_0/2)$. Then by {\it Step 1},  
$$U_L(x) \geq  U_0(-x+\xi_*-\hat{\varepsilon}+K\hat{\varepsilon}\hat{\delta})-\hat{\varepsilon}\hat{\delta} e^{-\mu t}  \  \ \hbox{ for }\  \  t\geq 0,\, \, x\in \R.  $$
Passing to the limit as $t\to +\infty$, we obtain 
$$U_L(x) \geq  U_0\left(-x+\xi_*-\frac{1}{2}\hat{\varepsilon}\right) \ \ \hbox{ for }  \  \ x\in\R,  $$
by the definition of $\hat{\delta}$ and the monotonicity of $U_0$. This contradicts the definition of $\xi_*$ in~\eqref{define-xi*}, and hence, $U_L$ cannot be stationary. Consequently, $c^*_L>0$.  The proof of Lemma~\ref{sign-rightleft} is complete. 
\end{proof}

\subsection{Arbitrary asymmetrical speeds}

We are now in a position to prove that any pair of two speeds $c_L \geq 0$ and $c_R \geq 0$ is admissible. As we mentioned above, this result is optimal because, by the sign property of pulsating wave speeds (see  \eqref{sign-property}), both speeds cannot have opposite signs. Moreover, the case of $c_L \leq 0$ and $c_R \leq 0$ can be retrieved by a simple change of variables (replacing $u$ by $1-u$).

Let $\chi\in C^1(\R^2)$ be the function given at the beginning of the previous subsection.  We will find that, by considering a reaction term of the type
$$f_0 (u) + \sigma_1 \chi (x,u) + \sigma_2 \chi ( -x,u),$$
varying the parameters $\sigma_1$ and $\sigma_2$ and up to some rescaling, one can achieve any pair of nonnegative speeds. This will in particular prove Theorem \ref{arb-speed-1D}.

We start by checking that the following equation
\begin{equation}\label{eq-sigma1-2}
\partial_t u = \partial_{xx}  u +  f_0 (u) + \sigma_1 \chi (x,u) + \sigma_2 \chi (-x,u) \,\, \hbox{ for }\, \, t\in\R,\,x\in\R,
\end{equation} 
is of the bistable type, provided that $\sigma_1$ and $\sigma_2$ are small enough.
\begin{lem}\label{speed-sigma-0-new}
There exists $\sigma_*>0$ such that for every $0<\sigma_1 ,\sigma_2 \leq \sigma_*$, the equation~\eqref{eq-sigma1-2} is spatially periodic and bistable in the sense of Definition~\ref{bistable} with period~1; in particular, each intermediate $1$-periodic steady state $\bar{u}$ of \eqref{eq-sigma1-2} with $0<\bar{u}<1$ is linearly unstable.
\end{lem}

This lemma follows from the proof of Lemma \ref{inter-unstable} by some minor modifications; therefore we omit the proof. Up to reducing $\sigma_*$ and without loss of generality, we can assume that it is the same positive constant as in Proposition~\ref{prop-zero-pos}.\\

In the discussion below, we consider the family of equations
\begin{equation}\label{eq-sigma1-2-bis}
\partial_t u = \partial_{xx} u +  f(x,u;\tau)\,\, \hbox{ for }\, \, t\in\R,\,x\in\R,
\end{equation} 
where 
$$f(x,u;\tau) := f_0 (u) + \sigma_* \chi (x,u) + \tau \sigma_* \chi (-x,u)  \,\, \hbox{ for }  \,\,   x\in\R,\, u\in \R,$$
which is increasing with respect to the parameter $\tau \in [0,1]$.

By Lemma~\ref{speed-sigma-0-new}, equation~\eqref{eq-sigma1-2-bis} is of the spatially periodic bistable type with period~$1$. Therefore, it follows from Theorem \ref{th:tw} that for every $\tau\in [0,1]$, equation~\eqref{eq-sigma1-2-bis}  admits a leftward pulsating wave $U_{\tau,L}$ with speed $c_{\tau,L}$ and a rightward pulsating wave $U_{\tau,R}$ with speed~$c_{\tau,R}$, where both $c_{\tau,L}$ and $c_{\tau,R}$ are nonnegative since $\int_{0}^1\int_{0}^1 f(x,u;\tau)dxdu>0$. 
Moreover, we have the following lemma. 

\begin{lem}\label{sign-speed-tau-new}
The following statements hold true:
\begin{itemize}
\item [$(i)$] $c_{\tau,L}>0$ for $\tau\in[0,1]$; 
\item [$(ii)$] $c_{\tau,R}>0$ for $\tau\in (0,1]$ and $c_{0,R}=0$;
\item [$(iii)$] $c_{1,L}=c_{1,R}$. 
\end{itemize}
Consequently, $U_{\tau,L}$ is unique up to time shifts for $\tau\in [0,1]$ and $U_{\tau,R}$ is unique up to time shifts for $\tau\in (0,1]$. 
\end{lem} 

\begin{proof}
Notice that any shift of $U_0(\cdot)$ is a strict subsolution of \eqref{eq-sigma1-2-bis} with $\tau\in (0,1]$, and any shift of $U_0(-\cdot)$ is a strict subsolution of \eqref{eq-sigma1-2-bis} with $\tau\in [0,1]$. Then,  the positivity of $c_{\tau,L}$ for $\tau\in[0,1]$ and of $c_{\tau,R}$ for $\tau\in (0,1]$ follows from the proof of Lemma \ref{sign-rightleft} by some obvious modifications.

The fact that $c_{0,R}=0$ follows directly from Lemma \ref{sign-rightleft} (notice that $f(x,u;0)= f_{\sigma_*}$), and statement $(iii)$ is an easy consequence of the fact that $f(x,u;1)$ is symmetric with respect to $x$.  The proof of Lemma \ref{sign-speed-tau-new} is thus complete. 
\end{proof}

The next lemma is concerned with the monotonicity and continuity of the wave speeds with respect to the parameter~$\tau$.  

\begin{lem}\label{speed-tau-monotone-continue} 
The functions $\tau\mapsto c_{\tau,L}$ and $\tau\mapsto c_{\tau,R}$ are nondecreasing and continuous in $\tau\in [0,1]$.  
\end{lem}

\begin{proof}
Due to our construction of $f(x,u;\tau)$ and Lemma \ref{speed-sigma-0-new}, all the conclusions in Section~\ref{sec:prelim} hold for the equation \eqref{eq-sigma1-2-bis} with $\tau\in [0,1]$.  In particular, since $f(x,u;\tau)$ is increasing in $\tau\in [0,1]$, it follows from Lemma \ref{speed-tau-monotone-hd}  that the speeds $c_{\tau,L}$ and $c_{\tau,R}$ are nondecreasing in $\tau$. 

It remains to show that $c_{\tau,L}$ and $c_{\tau,R}$ are continuous with respect to $\tau$. 
We only prove the continuity of $c_{\tau,R}$, as the proof for $c_{\tau,L}$ is similar (actually, the proof of the continuity of $c_{\tau,L}$ is even simpler, since by Lemma~\ref{sign-speed-tau-new}~$(i)$, $c_{\tau,L}$ does not change sign with respect to $\tau\in[0,1]$). Let $(\tau_n)_{n\in\N} \subset [0,1]$ be an arbitrary sequence with $\tau_n \to \tau_\infty$ as $n\to + \infty$ for some $\tau_\infty \in [0,1]$. It suffices to show that 
\begin{equation}\label{cLn-cLinfty}
c_{\tau_{n},L} \to c_{\tau_{\infty},L} \,\,\hbox{ as } \,\, n\to +\infty. 
\end{equation}
Notice that $f(\cdot,\cdot;\tau_n)\to f(\cdot,\cdot;\tau_{\infty})$ as $n\to +\infty$ in $C^1(\R^{2})$. 
If $c_{\tau_{\infty},L}=0$, then \eqref{cLn-cLinfty} follows directly from statement $(i)$ of Lemma \ref{speed-tau-continue-hd}. On the other hand, if $c_{\tau_{\infty},L}\neq 0$, then we have $c_{\tau_{\infty},L}>0$ and by Lemma~\ref{sign-speed-tau-new}~$(ii)$, $\tau_{\infty}$ must be positive. This in particular implies that  $\tau_n>\tau_\infty/2$ for all large $n\in\N$. By using Lemmas \ref{speed-tau-monotone-hd} and~\ref{sign-speed-tau-new}~$(ii)$, we see that for all large $n\in\N$, $c_{\tau_{n},L}$ is bounded away from $0$ by a positive constant, and hence, Lemma \ref{speed-tau-continue-hd} $(ii)$ immediately implies \eqref{cLn-cLinfty}.  This ends the proof of Lemma~\ref{speed-tau-monotone-continue}.  
\end{proof}

We are now ready to prove Theorem \ref{arb-speed-1D} by a rescaling argument. 

\begin{proof}[Proof of Theorem \ref{arb-speed-1D}]
Let $c_L$ and $c_R$ be any two nonnegative numbers. Without loss of generality, we can restrict ourselves to the case where $c_L \geq c_R$ (the case where $c_L \leq c_R$ can be treated identically). We may also assume that $c_L>0$, as in the case where $c_L = c_R=0$, one can simply choose $f(x,u)= f_0 (u)$ the balanced Allen-Cahn nonlinearity, and the desired result is automatically proved.

Let us choose 
$$\gamma:= \frac{c_R}{c_L}.$$
Clearly, $0\leq \gamma \leq 1$.  By Lemma \ref{speed-tau-monotone-continue}, the set $\{(c_{\tau,L},c_{\tau,R}):\tau\in [0,1]\}$ of the wave speeds of \eqref{eq-sigma1-2-bis} is a continuous curve connecting $(c_{0,L}, 0)$ and $(c_{1,L},c_{1,R})$ in the quadrant region $\mathbb{R}^+ \times \mathbb{R}^+$. Since $c_{0,L} >0$ by Lemma~\ref{sign-speed-tau-new}~$(i)$, and since both $c_{\tau,L}$ and $c_{\tau,R}$ are nondecreasing in $\tau\in[0,1]$ by Lemma~\ref{speed-tau-monotone-continue}, this curve is away from the origin. Combining this with the fact that $c_{1,L}=c_{1,R}$ (see Lemma~\ref{sign-speed-tau-new}~$(iii)$), one infers that, for the above $\gamma\in [0,1]$, there exists some $\tau\in [0,1]$ such that equation \eqref{eq-sigma1-2-bis} has a rightward pulsating wave $U_{\tau,R}(t,x)$ with speed $c_{\tau,R}$ and a leftward pulsating wave $U_{\tau,L}(t,x)$ with speed $c_{\tau,L}$, and that 
$$c_{\tau,R} = \gamma c_{\tau,L} \geq 0.$$
Recalling that $c_{\tau,L}>0$, one finds some $\nu > 0$ such that $c_L = \nu c_{\tau,L}$. 
It is then easily checked from the above that  $c_R = \nu c_{\tau,R}$. 

Let 
$$f(x,u):=\nu^2 f(\nu x,u;\tau). $$  
Then $f(x,u)$ is $1/\nu$-periodic in $x$, and $u\equiv 1$, $u\equiv 0$ are linearly stable steady states of~\eqref{eq:main}. Moreover, by Lemma~\ref{speed-sigma-0-new}, it is clear that any $1/\nu$-periodic steady state $\bar{u}$ of \eqref{eq:main} with $0<\bar{u}<1$ is linearly unstable. This means that, with this choice of the reaction term, equation~\eqref{eq:main} is spatially periodic and bistable in the sense of Definition~\ref{bistable}. It is also easily seen that $U_{R}(t,x):=U_{\tau,R}(\nu^2t,\nu x)$ is a rightward pulsating wave of \eqref{eq:main} with speed $c_R$, and $U_{L}(t,x):=U_{\tau,L}(\nu^2t,\nu x)$ is a leftward pulsating wave with speed $c_L$. The proof of Theorem \ref{arb-speed-1D} is thus complete.
\end{proof}

%%%%%%%%%%%%%%%%%%%%%%%%%%%%%%%%%%%%%%%%%%%%%%%%%%%%%%%%
%%%%%%%%%%%%%%%%%%%%%%%%%%%%%%%%%%%%%%%%%%%%%%%%%%%%%%%%

\section{Admissible speeds in higher dimensions}\label{sec:dd}

We now turn to the higher dimension $d>1$ and prove Theorem \ref{theo-hd}, that is, the speeds can be chosen arbitrarily (as long as they do not change signs) in some arbitrarily large number of directions. 

Let us first define several subsets of the unit sphere $\mathbb{S}^{d-1}$. 
Denote by $(e_i)_{1\leq i\leq d}$ the standard basis of $\R^d$. We define
$$\mathcal{S} := \mathbb{S}^{d-1} \cap \Q^d,$$
which is the set of unit vectors with rational coordinates. Equivalently, $e \in \mathbb{S}^{d-1}$ belongs to~$\mathcal{S}$ if, for each $1\leq i\leq d$, 
$$ e \cdot e_i \in \Q . $$
Moreover, the following lemma holds true. 
\begin{lem}\label{dense}
The set $\mathcal{S}$ is dense in $\mathbb{S}^{d-1}$.
\end{lem}

\begin{proof}
We first check that $\mathcal{S}$ is a dense subset of $\mathbb{S}^{d-1}$, which is a classical consequence of the stereographic projection. The stereographic projection is a one-to-one continuous mapping from $\mathbb{S}^{d-1} \setminus P$, where $P$ is some point on the sphere, say for instance $(1,0,\cdots,0)$, to a hyperplane. It is defined by
$$\Pi  (x_1,\cdots,x_n) = \left( \frac{x_2}{1-x_1} , \cdots , \frac{x_d}{1-x_1} \right).$$
In particular, it is also a one-to-one mapping from $\mathcal{S} \setminus P$ to $\Q^{d-1}$. It immediately follows that $\mathcal{S} \setminus P$ is dense in $\mathbb{S}^{d-1} \setminus P$. Equivalently, $\mathcal{S}$ is dense in $\mathbb{S}^{d-1}$. 
\end{proof}
Next, we define a more general set which may include some vectors in the unit sphere with irrational coordinates.   
For any vector $L = (L_1, \cdots , L_d)$ with $L_i>0$ for $1\leq i\leq d$, we define 
\begin{equation}\label{assume-hd-bisbis} 
\mathcal{S}_L := \{ \zeta \in \mathbb{S}^{d-1} \, : \ \exists\, M >0  , \  \forall \,1 \leq i \leq d, \ \   L_i \zeta \cdot e_i \in M \Z \}.
\end{equation}
Clearly, when $L= (1,\cdots,1)$, then $\mathcal{S} \subset \mathcal{S}_L$ and in particular $\mathcal{S}_L$ is dense in $\mathbb{S}^{d-1}$.

The purpose of this section is to prove that any finite set of speeds in directions belonging to $\mathcal{S}_L$ is admissible. More precisely:

\begin{theo}\label{main-prop-hd}
 Let $L= (L_1, \cdots , L_d)$ be any $d$-uplet of positive real numbers, $N\geq 2$ be any integer and  
$(\zeta_1,\zeta_2,\cdots,\zeta_N)$ be any $N$-uplet of different directions in $\mathcal{S}_L$. 
Then for any $N$-uplet $(c_1,\cdots,c_N)\in \R_+^N \cup \R_-^N$, there exists a spatially periodic bistable equation in the sense of Definition~\ref{bistable} such that for each $1\leq j\leq N$, 
$$c^*(\zeta_j)=c_j,$$
where $c^*(\zeta_j)$ is the pulsating wave speed in the direction $\zeta_j$. 
\end{theo}
As mentioned above, in the special case $L=(1,\cdots,1)$, the set $\mathcal{S}_L$ includes $\mathcal{S}$ and thus Theorem~\ref{theo-hd} is an easy corollary of the above result. Still it should be pointed out that Theorem~\ref{main-prop-hd} is more general than Theorem \ref{theo-hd}. For example, in the case of $d=2$, for any two distinct directions $\zeta_j:=(\cos\theta_j,\sin\theta_j) \in \mathbb{S}^1$, $j=1,2$, one can check that by choosing $L_1=\cos(\theta_1-\theta_2)$, $L_2=\sin(\theta_1-\theta_2)$, any pair of speeds $(c_1,c_2)\in \R_+^2\cup \R_-^2$ is admissible in the directions  $(\zeta_1,\zeta_2)$ provided that $\tan \theta_j \in \Q$, $j=1,2$. Similarly, the speeds can be chosen independently in the eight directions $(\pm 1, 0)$, $(0,\pm 1)$ and $(\pm 1/ \sqrt{2} , \pm 1/\sqrt{2})$, though the latter four involve irrational coordinates.

The strategy of the proof of Theorem \ref{main-prop-hd} is similar in spirit to that of Theorem \ref{arb-speed-1D}. We will construct a spatially periodic bistable equation by perturbing the homogeneous balanced bistable equation. Let us first show a lemma which will be useful to ensure that our construction below defines a spatially periodic equation.
\begin{lem}\label{lem:period_nd}
Take $\zeta \in \mathcal{S}_L$ and $M >0$ such that $ L_i \zeta \cdot e_i \in  M \Z$ for all $1 \leq i \leq d$.
If $(z,u) \in \R^2 \mapsto g(z,u)$ is a $M$-periodic function in its $z$-variable, then $(x,u) \in \R^d \times \R \mapsto g (x \cdot \zeta, u)$ is $L$-periodic in its $x$-variable.
\end{lem}
\begin{proof}
Notice that the existence of $M$ follows immediately from the definition of $\mathcal{S}_L$. Then, 
$$g( (x + L_1 e_1) \cdot \zeta  , u) =g (x \cdot \zeta + L_1  e_1 \cdot \zeta, u) = g (x\cdot \zeta, u),$$
where in the last equality we used the fact that $L_1  \zeta \cdot e_1 \in M \mathbb{Z}$. The periodicity in other directions can be checked similarly.
\end{proof}

%%%%%%%%%%%%%%%%%%%%%%%%%%%%%%%%%%%%%%%%%%%%%%%%%%%%%%%

\subsection{Construction of a bistable nonlinearity}
 
Let us now start the proof of Theorem~\ref{main-prop-hd}. From now on, we fix $L_1>0, \cdots, L_d >0$, $N \geq 2$ and let $(\zeta_1,\cdots \zeta_N)$ be distinct vectors in $\mathcal{S}_L$. In particular, due to~\eqref{assume-hd-bisbis}, there exist some positive real numbers~$M_1, \cdots, M_N$ such that 
\begin{equation*}\label{choice-M-bis}
L_i  \zeta_j \cdot e_i  \in M_j \Z \,\,\hbox{ for all }\, 1\leq i\leq d,\,  1\leq j\leq N. 
\end{equation*}
As in the one-dimensional case, the basis of our construction is the homogeneous balanced reaction $f_0$ given in \eqref{eq:balanced}. Recall that $U_0$ is the stationary traveling wave of \eqref{Allen-Cahn} normalized by $U_0(0) = 1/2$. Then for any $1\leq j\leq N$ and any shift $m\in\Z$, the function
$$x\mapsto U_0 \left( x\cdot\zeta_j+ m M_j \right)$$
gives a stationary wave of \eqref{Allen-Cahn} in the direction $\zeta_j$. 

For any integer $1\leq j\leq N$, let $E_{j}\subset\R \times [0,1]$ defined by
$$E_{j}=\left\{(z,u):\, U_0 \left( z \right) \leq u < U_0 \left(z - M_j\right) \right\}. $$
Proceeding similarly as in the definition of $\chi$ at the beginning of Section~\ref{sec:1d}, we can find a $C^1$ function $\chi_j:\R\times\R\to [0,1]$ such that it is $M_j$-periodic in its first variable, and that
$$ \chi_{j}(z,u)=0 \,\,\hbox{ for all } \,\,z \in\R,\,u\in (-\infty,\delta_0]\cup[1-\delta_0,\infty), $$
\begin{equation}\label{chi-eqv-0}
\chi_{j}=0\ \,\hbox{ in }\,\, \{(z,u)\in\R^2: \,(z-m M_j,u) \in \partial E_j \hbox{ for some } m\in \Z\},
\end{equation}
and
\begin{equation}\label{chi-large-0}
\chi_j>0 \ \,\hbox{ in }\,\, \{(z,u)\in\R^2: \, 2\delta_0\leq u\leq 1-2\delta_0, (z-m M_j,u) \in E_j\setminus \partial E_j \hbox{ for some } m\in\Z \}, 
\end{equation}
where $\delta_0\in (0,1/4)$ is a constant such that \eqref{choose-delta0} holds. 
Moreover, since $\zeta_j\in \mathcal{S}_L$,  Lemma~\ref{lem:period_nd} implies that the functions 
$$\{ (x,u ) \mapsto \chi_j (x \cdot \zeta_j , u) \}_{1 \leq j \leq N},$$
are $L$-periodic in $x$ with $L=(L_1,L_2,\cdots,L_d)$. 

Denote by $X$ the closed unit hypercube in $\R^N$, i.e. $X=[0,1]^N$.
For any $\sigma>0$ and any $\tau=(\tau_1,\cdots,\tau_N)\in X$, we define a reaction $f_{\sigma}(x,u;\tau)$ as follows:
\begin{equation*}
f_{\sigma}(x,u;\tau)=f_0(u)+\sigma\sum_{j=1}^N \tau_j\left(\prod_{l\neq j} \chi_l(x \cdot \zeta_l ,u) \right)\,\hbox{ for }\, x\in\R^d,\,u\in\R.
\end{equation*}
Clearly, $f_{\sigma}(x,u;\tau)$ is $L$-periodic in $x$,  $f_{\sigma}(\cdot,0;\tau)\equiv 0$, $f_{\sigma}(\cdot,1;\tau)\equiv 0$, and 
$$f_{\sigma}(x,u;\tau)\equiv f_0(u) \,\,\hbox{ for }\,\, u\in (-\infty,\delta_0]\cup [1-\delta_0,\infty).$$
This in particular implies that  $u\equiv 0$ and $u\equiv 1$ are linearly stable steady states of the equation
\begin{equation}\label{eq-sigma-tau}
\partial_t u =  \Delta u + f_{\sigma}(x,u;\tau)\,\,\hbox{ for }\,\,  t\in\R,\,x\in\R^d. 
\end{equation}
Furthermore, the following lemma ensures that the above equation is of the bistable type provided that $\sigma>0$ is small. 

\begin{lem}\label{exist-pf-hd}
There exists $\sigma_* >0$ such that for any $\tau\in X$, equation~\eqref{eq-sigma-tau} with $\sigma = \sigma_*$ is of the spatially periodic bistable type in the sense of Definition \ref{bistable}. Consequently, for any $e\in \mathbb{S}^{d-1}$, it has a pulsating wave connecting $0$ and $1$ in the direction~$e$, and its speed is unique.
\end{lem}
\begin{proof} 
The proof follows from that of Lemma \ref{inter-unstable} by some obvious modifications. 
\end{proof}

%%%%%%%%%%%%%%%%%%%%%%%%%%%%%%%%%%%%%%%%%%%%%%%%%%%%

\subsection{Properties of pulsating wave speeds}

In the discussion below, we fix $\sigma=\sigma_*$, where $\sigma_*$ is obtained in Lemma \ref{exist-pf-hd}. Then  equation~\eqref{eq-sigma-tau} becomes  
\begin{equation}\label{eq-sigma-tau-new}
\partial_t u =  \Delta u + f(x,u;\tau), \quad  t\in\R,\,\,x\in\R^d, 
\end{equation}
where $f(x,u;\tau):= f_{\sigma_*}(x,u;\tau)$.

Lemma \ref{exist-pf-hd} implies in particular that for any $\tau\in X$ and any $1\leq j\leq N$, there exists a unique speed $c^*(\zeta_j)$ such that \eqref{eq-sigma-tau-new} has a pulsating wave in the direction~$\zeta_j$ with speed $c^*(\zeta_j)$. Hereafter, we denote by $U_j$ this pulsating wave.
Since 
$$\int_{0}^1\int_{(0,L_1)\times \cdots \times(0,L_d)} f(x,u;\tau) dxdu \geq 0,$$
it follows from the sign property of the wave speed (see \eqref{sign-property}) that $c^*(\zeta_j) \geq 0$.

The first main result of this subsection is stated in the following lemma, which is concerned with the dependence of the sign of $c^*(\zeta_j)$ with respect to $\tau$. 

\begin{lem}\label{sign-speed-hd} 
For any $\tau=(\tau_1,\cdots,\tau_N)\in X$ and any integer $1\leq j\leq N$, we have 
\begin{itemize}
\item[$(i)$]
$c^*(\zeta_j)=0$ if $\tau_j=0$; 
\item[$(ii)$] 
$c^*(\zeta_j)>0$ if $\tau_j\neq0$. 
\end{itemize}
\end{lem}

The above lemma is analogous to Proposition \ref{prop-zero-pos} in dimension 1. Indeed, owing to our construction of $f$, we will show that the stationary wave $U_0(x\cdot\zeta_j)$ blocks the propagating in the direction $\zeta_j$ if $\tau_j= 0$, while any shift of $U_0(x\cdot\zeta_j)$ is a strict subsolution of \eqref{eq-sigma-tau-new} if $\tau_j \neq 0$, which forces propagation with positive speed in the direction $\zeta_j$. For clarity, let us first show the following lemma. 

\begin{lem}\label{sub-strict-sub}
For any $\tau=(\tau_1,\cdots,\tau_N)\in X$ and any integer $1\leq j\leq N$, the following statements hold true: 
\begin{itemize}
\item [$(i)$] if $\tau_j=0$, then $U_0 (x\cdot\zeta_j)$ is a stationary pulsating wave of \eqref{eq-sigma-tau-new} in the direction $\zeta_j$;

\item[$(ii)$] if $\tau_j\neq 0$, then for any $\xi\in\R$, $U_0 (x\cdot\zeta_j+\xi)$ is a strict subsolution of \eqref{eq-sigma-tau-new}, i.e. $U_0 (x\cdot\zeta_j+\xi)$ is a subsolution but not a solution. 
\end{itemize}
\end{lem}

\begin{proof}
If $\tau_j=0$, then by the construction of $f$, we have 
\begin{equation*}
f(x,u;\tau)=f_0(u)+\sigma_*\sum_{k\neq j} \tau_{k}\left(\prod_{l\neq k} \chi_{l}(x \cdot \zeta_l,u) \right)\,\hbox{ for }\, x\in\R^d,\,u\in[0,1].
\end{equation*}
Since $\chi_j(x \cdot \zeta_j ,U_j(x\cdot\zeta_j))\equiv 0$ because of \eqref{chi-eqv-0}, it is easily checked that $U_0(x\cdot \zeta_j)$ is a stationary wave of~\eqref{eq-sigma-tau-new}. 

Now we turn to the case where $\tau_j\neq 0$. Since each $\chi_l$, $1\leq l\leq N$, is nonnegative, it is easily seen that for any $\xi\in\R$, $U_0(x\cdot\zeta_j +\xi)$ is a subsolution of \eqref{eq-sigma-tau-new}. By the construction of $f$, to prove that $U_0(x\cdot\zeta_j +\xi)$ is a strict subsolution, it suffices to show that 
\begin{equation*}
\prod_{l\neq j} \chi_l(x \cdot \zeta_l , U_0(x\cdot\zeta_l +\xi)) \not \equiv  0.
\end{equation*}
In other words, one only needs to find some $a\in (0,1)$ such that 
\begin{equation}\label{not-subset}
\left\{x\in\R^d: U_0(x\cdot\zeta_j +\xi)=a\right\} \,\nsubseteq  \, \cup_{l\neq j} \left\{x\in\R^d: \chi_l(x \cdot \zeta_l ,a) =0\right\}.
\end{equation}
Notice that $U_0(\cdot)$ is decreasing. For any $a\in (2\delta_0,1-2 \delta_0)$ and any $1\leq l\leq N$, by \eqref{chi-eqv-0} and~\eqref{chi-large-0}, we see that $\chi_l(x \cdot \zeta_l,a) =0$ if and only if $x\cdot\zeta_l=U_0^{-1}(a)+mM_l$ for some  $m\in \Z$, where $U_0^{-1}$ is the inverse to $U_0$, and $\delta_0$ and $M_l$ are the constants given in the definition of~$\chi_l$. This means that $\{x\in\R^d: \chi_l(x \cdot \zeta_l,a) =0\}$ is a union of countably many hyperplanes in $\R^d$. By the same argument, the set $\{x\in\R^d: U_0(x\cdot\zeta_j +\xi)=a\}$ is a single hyperplane.

Now recall that $(\zeta_l)_{1\leq l\leq N}$ are different directions on $\mathbb{S}^{d-1}$. For any $1\leq l\leq N$ with $l\neq j$, we have either $\zeta_l\not \in \{ \zeta_j ,  -\zeta_j\}$, or $\zeta_l = -\zeta_j$. In the former case, 
\begin{equation*}
\{x\in\R^d: U_0(x\cdot\zeta_j +\xi)=a\}\,   \cap \,  \{x\in\R^d: \chi_l(x \cdot \zeta_l ,a) =0\}
\end{equation*}
is a union of countably many $(d-2)$-dimensional subspaces of $\R^d$. It follows that for any $a \in (2\delta_0, 1 -2 \delta_0)$, 
\begin{equation*}
\{x\in\R^d: U_0(x\cdot\zeta_j +\xi)=a\} \, \cap\, \cup_{l \neq j , \zeta_l \neq - \zeta_j } \{x\in\R^d: \chi_l(x \cdot \zeta_l ,a) =0\} \neq \emptyset ,
\end{equation*}
and more importantly it is a strict subset of $\{x\in\R^d: U_0(x\cdot\zeta_j +\xi)=a\}$. 

In the other case when $\zeta_l= -\zeta_j$, the set $\{x\in\R^d: U_0(x\cdot\zeta_j +\xi)=a\}$ is parallel to each hyperplane in $\{x\in\R^d: \chi_l(x \cdot \zeta_l, a) =0\}$. 
By choosing $a\in (2\delta_0,1- 2\delta_0)$ such that $U_0^{-1}(a) \not\in \{(\xi+mM_l)/2\}_{m\in\Z}$, one gets
$$ \{ x \in \R^d : U_0 (x \cdot \zeta_j + \xi) = a \} \cap \{ x \in \R^d : \chi_l (x \cdot \zeta_l, a) = 0 \} = \emptyset.$$
Combining the above, we can conclude that there exists $a\in (2\delta_0,1-2\delta_0)$ such that 
 \eqref{not-subset} holds true. Therefore, $U_0(x\cdot\zeta_j +\xi)$ is a strict subsolution of \eqref{eq-sigma-tau-new}.
\end{proof}

\begin{proof}[Proof of Lemma \ref{sign-speed-hd}]

Statement~$(i)$ follows directly from Lemma \ref{sub-strict-sub}~$(i)$ and the uniqueness of wave speeds of pulsating waves in a given direction (see Theorem \ref{th:tw}). The proof of statement~$(ii)$ is rather similar to that of Lemma \ref{sign-rightleft}. Therefore, we only give its outline and provide the details when considerable changes are needed. 

 Let us argue by contradiction and assume that $c^*(\zeta_j)=0$ for some $\tau$ with $\tau_j > 0$. Then there exists $U_j=U_j(x)$ a stationary pulsating wave in the direction~$\zeta_j$. Now let $\mu=- \max \{f_0'(0),\,f_0'(1)\}/2$ and $\delta_0\in (0,1/4)$ be the positive constant given in the definition of $(\chi_l)_{1\leq l\leq N}$. 
Then, by similar arguments to those used in {\it Step 1} of the proof of Lemma~\ref{sign-rightleft}, one finds some positive constant $K$ such that if for some $\varepsilon \in (0,\delta_0/2]$ and $\xi_0\in\R$, 
$$u_0(x)\geq  U_0(x\cdot\zeta_j+\xi_0)-\varepsilon \,\,\hbox{ for } \,\, x\in\R^d,$$
then $u(t,x;u_0)$ the corresponding solution of~\eqref{eq-sigma-tau-new} satisfies
\begin{equation}\label{sub-solu-new-hd} 
u(t,x;u_0) \geq  U_0(x\cdot\zeta_j+\xi_0+K\varepsilon)-\varepsilon e^{-\mu t}\,\, \hbox{ for }\,\,t\geq 0,\, x\in\R^d.
\end{equation}
By using this comparison result, one can conclude that (see {\it Step 2} of the proof of Lemma~\ref{sign-rightleft} and recall also Lemma~\ref{sub-strict-sub}~$(ii)$)
\begin{equation}\label{Uj-U0-hd}
U_j(x)> U_0(x\cdot\zeta_j+\xi_*)\,\hbox{ for all }\, x\in\R^d, 
\end{equation}
where $\xi_*$ is a real number defined by 
$$ \xi_*:=\inf\{\xi\in\R:\, U_j(x) \geq U_0(x\cdot\zeta_j+\xi)  \,\hbox{ for } \,x\in\R^d \}.$$
Take a large positive constant $C$ such that 
$$\hat{\delta}:= \sup_{|x\cdot\zeta_j|\geq C-1}  | U'_0(x\cdot\zeta_j+\xi_*) | \leq \frac{1}{2K}.$$

Now, we claim that there exists some constant $\hat{\varepsilon}\in (0,\delta_0/2)$ such that
\begin{equation}\label{Uj-U0-bound}
U_j(x) > U_0(x\cdot\zeta_j+\xi_*-\hat{\varepsilon}) \,\, \hbox{ for all }\, \, |x\cdot\zeta_j|\leq C.
\end{equation}
Suppose on the contrary that it is not true. Then due to \eqref{Uj-U0-hd}, one can find a sequence $(x_n)_{n\in\N} \subset \R^d$ such that 
$|x_n\cdot\zeta_j|\leq C$ for all $n\in\N$, and
\begin{equation}\label{Uj-shift-U_0}
\lim_{n\to +\infty} (U_j(x_n)-U_0(x_n\cdot\zeta_j+\xi_*))=0.
\end{equation}
For each $n\in\N$, write $x_n=x_n'+x_n''$ with $x_n' \in L_1\Z\times \cdots\times L_d\Z$ and $x_n'' \in [0,L_1]\times \cdots \times[0,L_d]$, and set $V_n(x)=U_j(x+x_n')$ for $x\in\R^d$. Clearly, the function $V_n(x)$ is a stationary pulsating wave of \eqref{eq-sigma-tau-new}, since $f$ is $L$-periodic in $x$. Up to extraction of some subsequence, we can assume that $x_n'' \to x_{\infty}$ and $x_n\cdot\zeta_j \to \xi_{\infty}$  as $n\to +\infty$ for some $x_{\infty}  \in [0,L_1]\times \cdots \times[0,L_d]$, $\xi_{\infty}\in [-C,C]$, and that from standard elliptic estimates, $V_n(x)\to V_{\infty}(x)$ locally uniformly in $x\in\R^d$, where $V_{\infty}$ is a stationary solution of \eqref{eq-sigma-tau-new}. It follows directly from \eqref{Uj-U0-hd} that 
$$V_{n}(x+x_n'') >  U_0((x+x_n)\cdot\zeta_j+\xi_*)\,\,\hbox{ for all }\,\, x\in\R^d,\, n\in\N.$$
Passing to the limit as $n\to +\infty$, we obtain $V_{\infty}(x+x_{\infty}) \geq  U_0(x\cdot\zeta_j+ \xi_{\infty}+\xi_*)$ for all $x\in\R^d$.  
Then the same reasoning as used in showing \eqref{Uj-U0-hd} implies that 
$$V_{\infty}(x+x_{\infty}) >  U_0(x\cdot\zeta_j+ \xi_{\infty}+\xi_*)\,\,\hbox{ for all }\,\, x\in\R^d.$$
However, by \eqref{Uj-shift-U_0}, we have $V_{\infty}(x_{\infty})=U_0(\xi_{\infty}+\xi_*)$, which is a contradiction. Therefore, we can find some constant $\hat{\varepsilon}\in(0,\delta_0/2)$ such that \eqref{Uj-U0-bound} holds true.

On the other hand, for all $|x\cdot\zeta_j|>C$, by \eqref{Uj-U0-hd} and the definition of $\hat{\delta}$, we have 
$$U_j(x) - U_0(x\cdot\zeta_j+\xi_*-\hat{\varepsilon})  \geq   -\max_{|x\cdot\zeta_j|\geq C}|U_0(x\cdot\zeta_j+\xi_*)  -U_0(x\cdot\zeta_j+\xi_*-\hat{\varepsilon})| \geq -\hat{\varepsilon}\hat{\delta}. $$
This together with \eqref{Uj-U0-bound} implies 
$$U_j(x) \geq  U_0(x\cdot\zeta_j+\xi_*-\hat{\varepsilon})-\hat{\varepsilon}\hat{\delta}  \, \,\hbox{ for all }\, \, x\in \R^d.$$
Then by using the comparison result \eqref{sub-solu-new-hd}, we obtain  
$$U_j(x) \geq  U_0(x\cdot\zeta_j+\xi_*-\hat{\varepsilon}+K\hat{\varepsilon}\hat{\delta})-\hat{\varepsilon}\hat{\delta} e^{-\mu t}  \, \,\hbox{ for }\, \, t\geq 0,\, x\in \R^d.  $$
Passing to the limit as $t\to +\infty$, we get a contradiction with the definition of $\xi_*$. Therefore,  $c^*(\zeta_j)$ must be positive.  The proof of Lemma~\ref{sign-speed-hd} is complete. 
\end{proof}

\smallskip

In the remaining part of this subsection, we consider the monotonicity and continuity of the speeds $c^*(\zeta_j)$ with respect to $\tau$. To indicate the dependence on $\tau$, we will write $c^*(\zeta_j,\tau)$ instead of $c^*(\zeta_j)$. However, if no confusion exists, we will again use $c^*(\zeta_j)$. 

For the convenience of our statement, let us define a (partial) order on $X = [0,1]^N$. For any two points $\tau^1=(\tau^1_1,\tau^1_2,\cdots, \tau^1_N)$ and $\tau^2=(\tau^2_1,\tau^2_2,\cdots,\tau^2_N)$ in $X$, we say 
\begin{equation}\label{define-order}
\tau^1\geq \tau^2 \,\, \hbox{ if and only if }\,\,  \tau^1_j\geq \tau^2_j \,\hbox{ for all }\,  1\leq j\leq N. 
\end{equation}

\begin{lem}\label{monotone-continue-hd} 
For any integer $1\leq j\leq N$, the following statements hold true:
\begin{itemize}
\item[$(i)$]  $c^*(\zeta_j,\tau^1) \geq c^*(\zeta_j,\tau^2)$ whenever $\tau^1\geq \tau^2$ in $X$; 

\item[$(ii)$]  $c^*(\zeta_j,\tau)$ is continuous with respect to $\tau\in X$.
\end{itemize}
\end{lem}

This lemma is an easy consequence of Lemmas \ref{speed-tau-monotone-hd}, \ref{speed-tau-continue-hd} and \ref{sign-speed-hd}. Since the verification is similar to the proof of Lemma \ref{speed-tau-monotone-continue}, the details are not repeated here.

%%%%%%%%%%%%%%%%%%%%%%%%%%%%%%%%%

\subsection{An intermediate mean value lemma}
 Recall that $X$ is the closed unit hypercube in $\R^N$ equipped with a (partial) order in the sense of  \eqref{define-order}. In this subsection, we show an intermediate mean value result in dimension~$N$, which will ensure that the set of speeds $\{( c^*(\zeta_1,\tau),\cdots, c^*(\zeta_N,\tau)): \tau\in X\}$ at least contains a small closed hypercube in $\R^N$.
\begin{lem}\label{IMV}
Let $G = (g_1,g_2,\cdots, g_N) : X \to \R^N$ be a continuous function. Assume that
for each $1\leq j\leq N$, 
\begin{equation}\label{assume-fk-1}
g_j(x)=0\,\hbox{ for all } \, x=(x_1,x_2,\cdots, x_N)\in X  \,\hbox{ with }\, x_j=0,  
\end{equation}
and 
\begin{equation}\label{assume-fk-2}
g_j(x)>0\,\hbox{ for all } \, x=(x_1,x_2,\cdots, x_N)\in X  \,\hbox{ with }\, x_j\neq 0.   
\end{equation}
Assume also that $G$ is order-preserving in the following sense:
$$  G (x) \geq G  (y)\, \hbox{ in }\, \R^N  \, \hbox{ whenever }\, x \geq y \, \hbox{ in }\,  X. $$
Then there exists a constant $\eta_*>0$ such that  $[0,\eta_*]^N \subset G(X)$. 
\end{lem}

We will show the above lemma by a topological degree argument. Let us first recall the definition of the topological degree in finite dimension. Let 
$$ 
\Gamma :=\left\{(G,\Omega,y): \,\Omega\subset \R^N \hbox{ open bounded}, \, G\in C(\bar{\Omega}; \R^N) \hbox{ and } \, y\in \R^N \backslash G(\partial \Omega) \right\}. 
$$
An integer valued function $deg:=deg(G,\Omega,y)$ is called a topological degree associated with $(G,\Omega,y) \in \Gamma$ if it satisfies the following properties:
\begin{itemize}
\item (Normalization) $deg(I,\Omega,y)=1$ for $y\in\Omega$, where $I$ denotes the identity map of $\R^N$;   

\item (Additivity) $deg(G,\Omega,y)=deg(G,\Omega_1,y)+deg(G,\Omega_2,y)$ whenever $\Omega_1$, $\Omega_2$ are disjoint open subsets of $\Omega$ such that $y\not \in G(\bar{\Omega}\backslash (\Omega_1 \cup \Omega_2))$;

\item (Homotopy Invariance) $deg(H(s,\cdot),\Omega,y(s))$ is independent of $s\in [0,1]$ whenever $H:[0,1]\times \bar{\Omega}\to \R^N$ is continuous, $y:[0,1]\to \R^N$ is continuous and $y(s) \not\in H(s,\partial\Omega)$ for all $s\in[0,1]$. 
\end{itemize}

It is well known (see e.g., \cite[Chapter 1]{Deim}) that there is a unique function $deg: \Gamma \to \Z$  satisfying the above conditions. Furthermore, these conditions also imply:

\begin{lem}\label{degree}
\begin{itemize}
\item[$(i)$] If $deg(G,\Omega,y)\neq 0$, then there exists some $x\in\Omega$ such that $G(x)=y$;

\item[$(ii)$] $deg(A,\Omega,y)={\rm sign} (\det A)$ for linear maps $A$ with $\det A \neq 0$ and $y\in A(\Omega)$. 
\end{itemize}
\end{lem}

\begin{proof}
See \cite[Theorems 1.1.1 and 1.3.1]{Deim}. 
\end{proof}

\begin{proof}[Proof of Lemma \ref{IMV}]
We first take a linear map $G_0: X\to\R^N$ by $G_0(x)=Ax$ for $x\in X$, 
where $A$ is a diagonal matrix given by 
$$A=diag(g_1(\hat{e}_1),g_2(\hat{e}_2),\cdots, g_N(\hat{e}_N)). $$ 
Here, $(\hat{e}_j)_{1\leq j\leq N}$ denotes the standard basis of $\R^N$. 
By the assumption \eqref{assume-fk-2}, we have $g_j(\hat{e}_j)>0$ for all $1\leq j\leq N$. It then follows directly from Lemma \ref{degree}~$(ii)$ that 
$$ deg(G_0,X_0,y)= 1 \, \hbox{ for all } \,y\in  (0,\eta_*)^N. $$
where $\eta_* :=\min_{1\leq j\leq N}g_j(\hat{e}_j)$, and  $X_0$ denotes the open unit hypercube in $\R^N$, i.e. $X_0=(0,1)^N$.

We now define the function $H:[0,1]\times X \to \R^N$ by 
$$H(s,x):=(1-s)G_0(x)+sG(x) \,\hbox{ for }\, (s,x) \in  [0,1]\times X.$$
Clearly, $H(s,x)$ is a continuous function, and $H(0,x)\equiv G_0(x)$, $H(1,x)\equiv G (x)$. Let us write $H(s,x)=(H_1(s,x),\cdots, H_N(s,x))$. 
Due to the assumption \eqref{assume-fk-1}, it is also easily checked that for each $1\leq j\leq N$, 
\begin{equation}\label{boundary-1}
H_j(s,x)=0\,\hbox{ for all } \,s\in [0,1],\,  x=(x_1,x_2,\cdots, x_N)\in X  \,\hbox{ with }\, x_j=0.
\end{equation}
Moreover, since $G$ is order-preserving, it follows that for each $1\leq j\leq N$, 
$$g_j(x)\geq g_j(\hat{e}_j)  \,\hbox{ for all }  \, x=(x_1,x_2,\cdots, x_N) \in X  \,\hbox{ with }\, x_j=1.$$
This implies that
$$H_j(s,x)\geq g_j(\hat{e}_j)\,\hbox{ for all } \,s\in [0,1],\,  x=(x_1,x_2,\cdots, x_N)\in X  \,\hbox{ with }\, x_j=1.$$
Combining this with \eqref{boundary-1}, we obtain 
$$H(s,\partial X) \cap  (0,\eta_*)^N =\emptyset\,\hbox{ for all } \,s\in[0,1]. $$
Then, by the homotopy invariance of the topological degree, we have 
$$ deg(G,X_0,y)= 1 \, \hbox{ for all } \,y\in  (0,\eta_*)^N. $$
This together with Lemma \ref{degree}~$(i)$ implies that $(0,\eta_*)^N \subset  G (X_0)$. Finally, by the closeness of~$X$ and the continuity of $G$, we can conclude that $[0,\eta_*]^N \subset  G(X)$. This ends the proof of Lemma~\ref{IMV}. 
\end{proof}

\subsection{Proof of Theorem \ref{main-prop-hd}}

In view of Lemma \ref{exist-pf-hd}, we can define a function $G: X \to  \R_+^N$ by
$$G (\tau)=(c^*(\zeta_1,\tau),\,\cdots,\, c^*(\zeta_N,\tau)).$$
It is easily seen from Lemmas~\ref{sign-speed-hd}, \ref{monotone-continue-hd} above that the function $G$ satisfies all the conditions in Lemma \ref{IMV}. As a consequence, there exists $\eta_*>0$ such that $[0,\eta_*]^N \subset G(X)$. This immediately gives the following result. 

\begin{lem}\label{speed-bound}
Let $(c_1,\cdots,c_N)$ be any $N$-uplet such that $0\leq c_j\leq \eta_*$ for all $1\leq j\leq N$. Then there exists $\tau\in X$ such that for any $1\leq j\leq N$, equation \eqref{eq-sigma-tau-new} has a pulsating wave in the direction $\zeta_j$ with speed $c_j$. 
\end{lem}

Having in hand the above lemma, we are now in a position to prove Theorem \ref{main-prop-hd}. As mentioned earlier, it immediately implies Theorem \ref{theo-hd}. 

\begin{proof}[Proof of Theorem \ref{main-prop-hd}]
The proof follows from a rescaling argument similar to that used in the last part of the proof of Theorem \ref{arb-speed-1D}. Without loss of generality, we assume that  $(c_1,\cdots,c_N)\in \R_+^N$ (the case where $(c_1,\cdots,c_N)\in \R_-^N$ can be retrieved by replacing $u$ by $1-u$).

Let $1\leq j_0\leq N$ be the integer such that 
$c_{j_0} =  \max \{c_j:  1\leq j\leq N\}$. Without loss of generality, we may assume that $c_{j_0}>0$, as in the case where $c_{j}=0$ for all $1\leq j\leq N$, the desired result is automatically proved by choosing  $f(x,u)= f_0 (u)$. 
Let us then set 
$$ \tilde{c}_j:= \eta_*\frac{c_j}{c_{j_0}}  \,\,\hbox{ for }\,\, 1\leq j\leq N,  $$
where $\eta_*$ is the positive constant provided by Lemma \ref{speed-bound}. It is clear that $0\leq \tilde{c}_j\leq \eta_*$ for all $1\leq j\leq N$. Consequently, there exists $\tau\in X$ such that for any $1\leq j\leq N$, equation~\eqref{eq-sigma-tau-new} has a pulsating wave $\tilde{U}_j(t,x)$ in the direction $\zeta_j$ with speed $\tilde{c}_j$. 

Set $\nu=c_{j_0}/\eta_*$ and 
\begin{equation*}%\label{final-f}
f(x,u)=\nu^2 f (\nu x,u;\tau), 
\end{equation*}
where $f(x,u;\tau)$ is the reaction term of equation \eqref{eq-sigma-tau-new}. Now we consider equation \eqref{eq:main} with the above reaction $f(x,u)$. It is easily checked that $f(x,u)$ is $L/\nu$-periodic in $x$, $u\equiv 1$ and $u\equiv 0$ are linearly stable steady states of~\eqref{eq:main}, and \eqref{eq:main} is of the bistable type in the sense of Definition~\ref{bistable}. Finally, notice that for any $1\leq j\leq N$, 
$$U_{j}(t,x):=\tilde{U}_{j}(\nu^2t,\nu x)$$ 
is a pulsating wave of \eqref{eq:main} with speed $\nu \tilde{c}_j = c_j$ in the direction $\zeta_j$. In other words, $c^*(\zeta_j)=c_j$. This ends the proof of Theorem \ref{main-prop-hd}.
\end{proof}

%%%%%%%%%%%%%%%%%%%%%%%%%%%%%%%%%%%%%%%%%%%%%%%%%%%%%%%
%%%%%%%%%%%%%%%%%%%%%%%%%%%%%%%%%%%%%%%%%%%%%%%%%%%%%%%%

\section{Asymmetrical terraces}\label{sec:terraces}

As already shown in Theorem \ref{arb-speed-1D}, there exists a spatially periodic bistable equation such that it has a rightward pulsating wave and a leftward pulsating wave moving with distinct speeds. In this section, we prove Theorem \ref{nonsysm-terrace} and construct a multistable equation by stacking finitely many such bistable equations with a common period. The key point is to make the speeds of these bistable pulsating waves ordered in different ways in the two opposite directions. As a consequence, different limiting states of these waves may be selected as platforms of the terraces in each direction.

\subsection{Preliminaries: Known results on terraces}

Let us first collect some properties on the existence and uniqueness of propagating terraces. These properties (which are only stated below for rightward propagating terraces but also hold in the left direction by a straightforward symmetry argument) will be used in the proof of Theorem \ref{nonsysm-terrace}.

\begin{lem}[{\cite[Theorem 1.8 and Corollary 1.9]{gm}}]\label{decom-terrace}
Assume that \eqref{eq:main} is of the spatially periodic and multistable type in the sense of Definition~\ref{def:multistable}. Then there exists a rightward propagating terrace $((U_k, c_k)_{1\leq k\leq N},(q_k)_{0\leq k\leq N})$ connecting~$0$ and~$p$, such that
$$\{q_k \, : \ 1\leq k\leq N \} \subset \{p_k  \,  : \ 1\leq k\leq  I  \}. $$
\end{lem}
The next lemma can be proved by the same argument as~\cite[Theorem 1.14]{gm}.
\begin{lem}\label{terrace-minimal}
Assume that \eqref{eq:main} is of the spatially periodic multistable type in the sense of Definition~\ref{def:multistable}, and let $\mathcal{T} := ((U_k, c_k)_{1\leq k\leq N},(q_k)_{0\leq k\leq N})$ be the propagating terrace from Lemma~\ref{decom-terrace}. If all the speeds $(c_k)_{1\leq k\leq N}$ are non-zero, then $\mathcal{T}$ is the unique (up to time shifts) rightward propagating terrace. 
\end{lem}
The third lemma is a consequence of \cite[Theorems 1.8 and 4.1]{gm} (see also~\cite[Definitions~1.3 and~1.4]{gm}):
\begin{lem}\label{steep-speed-bis}
Assume that \eqref{eq:main} is of the spatially periodic multistable type in the sense of Definition~\ref{def:multistable}, and let $\mathcal{T}:= ((U_k,c_k)_{1 \leq k \leq N}, (q_k)_{0 \leq k \leq N})$ be the propagating terrace from Lemma~\ref{decom-terrace}. Let also $p$ be a linearly stable steady state of \eqref{eq:main}, and $U$ be a pulsating wave connecting any of the $q_k$ and $p$ with speed $c$.

If $q_k >p$, then $c_{k+1} \leq c $, and furthermore, if $c =c_{k+1} \neq 0$, then $U$ is equal to~$U_{k+1}$ up to a time shift.

Similarly,  if $q_k < p$, then $c_k \geq c$, and furthermore, if $c =c_k \neq 0$, then $U$ is equal to~$U_k$ up to a time shift.
\end{lem}

%%%%%%%%%%%%%%%%%%%%%%%%%%%%%%%%%%%%%%%%%%%%%%%%%%%%%%%%%

\subsection{Proof of Theorem~\ref{nonsysm-terrace}}

We are now ready to prove Theorem \ref{nonsysm-terrace}. Let us first point out that in the special case $N=1$, Theorem \ref{nonsysm-terrace} is easily ensured by Theorem \ref{arb-speed-1D}, as the terraces in both directions consist of a single pulsating wave. In what follows, we give the proof in the general case $N\geq 2$. 

Let $(c_{R,k})_{1\leq k\leq I}$ and $(c_{L,k})_{1\leq k\leq I}$ be two arbitrary sequences of positive constants which will be required to satisfy certain order conditions later. By Theorem \ref{arb-speed-1D}, for each $1\leq k\leq I$, there exists a reaction $f_k(x,u)$ such that the following equation
\begin{equation}\label{eq:main-k}
\partial_t u = \partial_{xx} u +f_k(x,u) \ \ \hbox{ for }\ \ t\in\R,\,x\in\R,
\end{equation}
is spatially periodic and bistable, has a rightward pulsating wave $V_{R,k}(t,x)$ with speed $c_{R,k}$ and a leftward pulsating wave $V_{L,k}(t,x)$ with speed $c_{L,k}$, and both waves connect~$0$ and~$1$. Let us point out that the functions $f_k$ may have different periods; hence for each $k$, we denote by $L_k$ the period of $f_k$ with respect to its first variable.

Moreover, one can observe from the proof of Theorem \ref{arb-speed-1D} that for each $1\leq k\leq I$,  the function$f_k(x,u)$ can be chosen as homogeneous when $u$ is close to $0$ or $1$, satisfying
\begin{equation}\label{property-fk}
f_k(x,0)\equiv  f_k(x,1)\equiv 0,  \quad  \partial_u f_k(x,0)\equiv \partial_u f_k(x,1)\equiv \gamma, 
\end{equation}
where $\gamma$ is a negative constant (independent of $k$), and that up to time shifts, $V_{R,k}(t,x)$ ({\it resp.} $V_{L,k}(t,x)$) is the unique rightward ({\it resp.} leftward) pulsating wave of \eqref{eq:main-k} connecting~$0$ and~$1$. 

Now we define the reaction $g$ as
\begin{equation*}%\label{terrace-g}
g(x,u):= f_k(x,u-(I-k))\,\ \hbox{ for }\, x\in\R,\, I-k<u\leq I-k+1,
\end{equation*}
for $1\leq k\leq I$. Then we have the following lemma:
\begin{lem}\label{lem:rational_multistable}
Assume that $L_k \in \mathbb{Q}$ for all $1 \leq k \leq I$.
Then the equation
\begin{equation}\label{eq:main-stack}
\partial_t u = \partial_{xx} u + g(x,u) \ \ \hbox{ for }\ \ t\in\R,\,x\in\R,
\end{equation}
is of the spatially periodic multistable type in the sense of Definition~\ref{def:multistable}. 
\end{lem}
\begin{proof}
First, the function $g$ is clearly of class $C^1$, and it is also $L$-periodic with respect to its first variable, where $L$ is the smallest positive number such that $L/L_k \in \mathbb{N}$ for any $1 \leq k \leq I$, which exists thanks to the fact that the $L_k$ are rational numbers.

Next, the constant steady states $p_{k}:=I -k$ are linearly stable (regardless of the choice of the period). Moreover, by construction, the restriction of \eqref{eq:main-stack} to any interval $u \in [I-k , I-k+1]$ with $1 \leq k \leq I$ is of the bistable type, i.e. any $L_k$-periodic state strictly between $p_k$ and $p_{k-1}$ is linearly unstable. In particular, there exists a pulsating wave connecting $p_k$ and $p_{k-1}$; notice that, since $L$ is a multiple of~$L_k$, changing the period in Definition~\ref{pulsating-tw} is inconsequential.
\end{proof}
In particular, by Lemmas~\ref{decom-terrace} and~\ref{terrace-minimal}, equation~\eqref{eq:main-stack} admits a unique propagating terrace in the right and left directions. Furthermore, from the above proof, we have for each $1\leq k\leq I$ that
$$U_{R,k}(t,x):=V_{R,k}(t,x)+I-k \quad ({\it resp.} \,\,\, U_{L,k}(t,x):=V_{L,k}(t,x)+I-k)$$ 
is a rightward ({\it resp.} leftward) pulsating wave of \eqref{eq:main-stack} connecting ${p}_{k-1}$ and ${p}_k$ with speed $c_{R,k}$ ({\it resp.} $c_{L,k}$). Furthermore, whenever $c_{R,k}$ ({\it resp.} $c_{L,k}$) is non-zero, it is the unique (up to time shifts) pulsating wave connecting $p_{k-1}$ and $p_k$. 

By requiring the sequences of speeds $(c_{R,k})_{1\leq k\leq I}$ and $(c_{L,k})_{1\leq k\leq I}$ to be ordered differently, we will show that equation \eqref{eq:main-stack} has various asymmetrical terraces connecting $0$ and~$p_0$ in the two opposite directions. For clarity, we divide the remaining proof into three steps, and in each step, we show one statement of Theorem~\ref{nonsysm-terrace}. We point out that, while the propagating terraces obtained below are all associated with equation \eqref{eq:main-stack} and thus connect~$0$ and $p_0$, one can immediately replace $p_0 = I$ by 1 thanks to the simple change of variables $\tilde{u}= u/I$.\smallskip

{\it Step 1: Proof of statement $(i)$.} Let the above sequences of speeds $(c_{R,k})_{1\leq k\leq I}$ and $(c_{L,k})_{1\leq k\leq I}$ satisfy 
$$0<c_{R,1}<c_{R,2}<\cdots<c_{R,I},\quad 0<c_{L,1}<c_{L,2}<\cdots<c_{L,I},$$ 
and 
\begin{equation}\label{speed-require-1}
c_{R,k}\neq c_{L,k}\,\hbox{ for all } \, 1\leq k\leq I.
\end{equation}
Without loss of generality, we can assume that the periods $L_k$ are rational,  thanks to the following lemma: 
\begin{lem}\label{continuity-L}
For any $1\leq k\leq I$ and any small $\delta>0$, the following equation 
\begin{equation*}\label{eq:approx_period}
\partial_t u = \partial_{xx} u + f_k\left( \frac{x}{1+\delta} ,u \right)\,\hbox{ for }\, t\in\R,\,x\in\R,
\end{equation*}
is also of the spatially periodic bistable type (with period $(1+\delta) L_k$), and it has a rightward pulsating wave $V_{R,k,\delta}(t,x)$ with speed $c_{R,k,\delta}$ and a leftward pulsating wave $V_{L,k,\delta}(t,x)$ with speed $c_{L,k,\delta}$. Moreover, the pair of speeds $(c_{L,k,\delta}, c_{R,k,\delta})$ converges to $(c_{L,k}, c_{R,k})$ as $\delta\to 0$.
\end{lem}
\begin{proof}
Recall that $f_k(x,u)$ is homogeneous when $u$ is close to $0$ or $1$, and satisfies \eqref{property-fk}. One easily checks that $0$ and $1$ are uniformly (in $x$) stable zeros of $f_k(x,\cdot)$ in the sense that: there exists $\sigma\in (0,\frac{1}{2})$ such that $f_k(x,u)\leq -\frac{\gamma}{2}u$ for all $(x,u)\in \R\times [0,\sigma]$ and $f_k(x,u)\geq \frac{\gamma}{2}(1-u)$ for all $(x,u) \in \R\times[1-\sigma,1]$. Then,  \cite[Theorem 1.8]{dhz1} immediately implies 
Lemma \ref{continuity-L}.
\end{proof}

As a consequence, Lemma~\ref{lem:rational_multistable} applies and~\eqref{eq:main-stack} is multistable in the sense of Definition~\ref{def:multistable}. Then, according to Definition \ref{terrace}, $(({U}_{R,k},c_{R,k}),({p}_k))$ ({\it resp.} $(({U}_{L,k},c_{L,k}),({p}_k))$) is a propagating terrace of~\eqref{eq:main-stack} connecting~$0$ and ${p}_{0}$ in the right ({\it resp.} left) direction. It further follows from Lemma~\ref{terrace-minimal} that they are the unique propagating terrace connecting~$0$ and $p_{0}$ in each direction. Clearly, these two terraces share the same platforms, but because of \eqref{speed-require-1}, the speeds of pulsating waves are different. This ends the proof of statement~$(i)$.\smallskip

{\it Step 2: Proof of statement~$(ii)$.} We first show that, by letting $I=3$ and the speeds $(c_{R,k})_{1\leq k\leq 3}$, $(c_{L,k})_{1\leq k\leq 3}$ satisfy 
\begin{equation}\label{speed-require-2}
0<c_{R,2}<c_{R,1}<c_{R,3},\quad 0<c_{L,1}<c_{L,3}<c_{L,2},
\end{equation} 
equation \eqref{eq:main-stack} has a rightward terrace and a leftward terrace connecting $0$ and $p_{0}$, and both terraces consist of two pulsating waves, but they have different intermediate platforms.  

By Lemma \ref{continuity-L}, we can again assume without loss of generality that the periods $L_k$ are rational so that Lemma~\ref{lem:rational_multistable} applies and~\eqref{eq:main-stack} is multistable in the sense of Definition~\ref{def:multistable}. By Lemmas~\ref{decom-terrace} and~\ref{terrace-minimal}, there exist a unique rightward terrace $((U'_{R,k},c'_{R,k}),(q_{R,k}))$ and a unique leftward terrace $((U'_{L,k},c'_{L,k}),(q_{L,k}))$ of \eqref{eq:main-stack} connecting $0$ and ${p}_0$, and
\begin{equation}\label{q-RL-k}
\{q_{R,k}\} \subset \{{p}_k \},\quad   \{q_{L,k}\} \subset \{{p}_k \}.
\end{equation}
We claim that
\begin{equation}\label{platform-right}
q_{R,1}={p}_2\quad \hbox{ and }\quad  q_{R,2} = {p}_3.
\end{equation}
Assume by contradiction that \eqref{platform-right} is not true. It then follows from \eqref{q-RL-k} that two cases may happen: either $q_{R,1}={p}_3$ or $q_{R,1}={p}_1$. 

If the former case happens, then the rightward terrace connecting $0$ and ${p}_0$ is a single front $U'_{R,1}$. By Lemma~\ref{steep-speed-bis}, one infers that $c_{R,3} < c'_{R,1} < c_{R,1}$, which is a contradiction with the order relations of speeds $(c_{R,k})$ stated in \eqref{speed-require-2}. 

In the latter case, due to $c_{R,2}<c_{R,3}$ and by using Lemma \ref{steep-speed-bis} again, one can check that $q_{R,2}={p}_2$ and $q_{R,3}={p}_3$. This means that the rightward terrace connecting $0$ and ${p}_0$ consists of three pulsating waves. Recall that, for each $1\leq k\leq 3$, the function ${U}_{R,k}$ is the unique (up to time shifts) pulsating wave connecting ${p}_{k}$ and ${p}_{k-1}$; hence 
$$c'_{R,k}=c_{R,k} \,\,\hbox{ for }  \, 1\leq k\leq 3. $$
Due to \eqref{speed-require-2}, this is a contradiction with the fact that the speeds of a terrace must be ordered. Therefore, \eqref{platform-right} is proved. In the left direction, one can proceed analogously as above to get that 
$$q_{L,1}={p}_1\quad \hbox{ and }\quad  q_{L,2} = {p}_3 .$$
Now we can conclude that both $((U'_{R,k},c'_{R,k}),(q_{R,k}))$ and $((U'_{L,k},c'_{L,k}),(q_{L,k}))$ consist of two pulsating waves, but the intermediate platforms (respectively $p_2$ and $p_1$) are different. Therefore, we have obtained a situation where statement $(ii)$ of Theorem~\ref{nonsysm-terrace} holds with $N=2$. Another use of Lemma~\ref{steep-speed-bis} also implies that
$$c'_{R,1} \in  (c_{R,2}, c_{R,1} ), \quad c'_{L,2} \in ( c_{L,3} , c_{L,2} ) ,$$
which turns out to be useful in the induction sketched below.

Indeed, for any integer $N \geq 3$, by using an induction argument and repeating the above analysis, one can find a reaction $g(x,u)$ such that \eqref{eq:main-stack} has rightward and leftward terraces connecting $0$ and $p_0$ with $N$ pulsating waves, but they do not share any intermediate platform. For instance, when $N=3$, one must choose $I=5$ and, in addition to~\eqref{speed-require-2}, assume that $c_{R,5} > c_{R,3} > c_{R,4} > c_{R,1}$ and $c_{L,4} > c_{L,5} > c_{L,2}$. More generally, to deal with the case $N+1$, one must increase $I$ by 2 and choose $c_{R,k}$ and $c_{L,k}$ (with $k=I+1, I+2$) so that $c_{R,I +2} > c_{R,I} > c_{R,I+1} >\max_{k \leq I-1} c_{R,k}$ and $c_{L,I+1} > c_{L,I+2} > \max_{k \leq I} c_{L,k}$. The details are omitted.  \smallskip

{\it Step 3: Proof of statement~$(iii)$.} We choose $I=N$ and let the sequences of speeds $(c_{R,k})_{1\leq k\leq N}$, $(c_{L,k})_{1\leq k\leq N}$ satisfy 
\begin{equation}\label{speed-require-3}
c_{R,1}>c_{R,2}>\cdots> c_{R,N}>0\quad\hbox{and}\quad 0<c_{L,
1}<c_{L,2}<\cdots<c_{L,N}.
\end{equation}
As in the first two steps, we also assume without loss of generality that the periods $L_k$ are all rational numbers. Then by Lemma~\ref{lem:rational_multistable}, we get that~\eqref{eq:main-stack} is multistable in the sense of Definition~\ref{def:multistable}. Next, the same reasoning as used in the proof of {\it Step 1} implies that $(({U}_{L,k},c_{L,k}),({p}_k))$ is the unique terrace of  \eqref{eq:main-stack} connecting $0$ and ${p}_0$ in the left direction. From Lemmas~\ref{decom-terrace} and~\ref{terrace-minimal}, there exists a unique terrace $((U'_{R,k},c'_{R,k})_{1\leq k\leq M},(q_{R,k})_{1\leq k\leq M})$ connecting 0 and~${p}_0$ in the right direction, and
$$(q_{R,k})_{1\leq k\leq M} \subset  (p_k)_{1\leq k\leq N},$$
where $M$ is a positive integer less than $N$.   

Next, we claim that $M=1$, that is, the rightward terrace of \eqref{eq:main-stack} connecting $0$ and~${p}_0$ is a single wave. Assume by contradiction that $M \geq 2$. Then $q_{R,1} = p_k$ for some $1 \leq k \leq N-1$, and in particular $U'_{R,1}$ connects $p_k$ and $p_0$. By Lemma~\ref{steep-speed-bis}, we find that $c'_{R,1} \geq c_{R,k}$. Similarly, we also get that $c'_{R,2} \leq c_{R,k+1}$. Recalling~\eqref{speed-require-3}, we obtain that $c'_{R,2} < c'_{R,1}$, which contradicts the order condition on the speeds of a propagating terrace. Thus $M=1$ and the proof of Theorem~\ref{nonsysm-terrace} is complete. 

\section*{Acknowledgments} 
This work was initiated when the authors visited Meiji University, in the framework of the International Research Network ReaDiNet jointly supported by CNRS and Meiji University.

The authors would like to thank Professor Xing Liang for suggesting the issue of one-dimensional admissible speeds tackled in Theorem~\ref{arb-speed-1D}, and Professors Yihong Du and Fran{\c c}ois Hamel for many helpful discussions.


\begin{thebibliography}{AAA}

\bibitem{ag} Alfaro, M., Giletti, T., {\it Varying the direction of propagation in reaction-diffusion equations in periodic media}, Netw. Heterog. Media 11 (2016), 369-393.

\bibitem{aw} Aronson, D. G., Weinberger, H. F., {\it Multidimensional nonlinear diffusion arising in population
genetics}, Adv. Math.  30 (1978), 33-76.

\bibitem{bh1} Berestycki, H., Hamel, F., {\it Front propagation in periodic excitable media}, Commun. Pure Appl. Math., 55 (2002), no. 8, 949-1032. 

\bibitem{bh2} Berestycki, H., Hamel, F., {\it Generalized transition waves and their properties}, Commun. Pure Appl. Math., 65 (2012), 592-648. 

\bibitem{Deim} Deimling, K., {\it Nonlinear Functional Analysis}, Springer-Verlag, Berlin-New York, 1985.

\bibitem{dhz1} Ding, W., Hamel, F., Zhao, X.-Q., {\it Bistable pulsating fronts for reaction-diffusion equations in a periodic habitat}, Indiana Univ. Math. J. 66 (2017), 1189-1265.

\bibitem{dhz2} Ding, W., Hamel, F., Zhao, X.-Q., {\it Propagation phenomena for periodic bistable reaction-diffusion equations}, Calc. Var. Part. Diff. Equations 54 (2015), 2517-2551.

\bibitem{dm2} Ding, W., Matano, H., {\it Dynamics of time-periodic reaction-diffusion equations with front-like initial data on $\R$}, SIAM J. Math. Anal. 52 (2020), no. 3, 2411-2462.

\bibitem{dum2} Du, Y., Matano, H., {\it Radial terrace solutions and propagation profile of multistable reaction-diffusion equations over $\R^N$}, arXiv 1711.00952, 2017, preprint.

\bibitem{ducasse} Ducasse, R., {\it Propagation properties of reaction-diffusion equations in periodic domains}, Analysis and PDE, to appear.

\bibitem{dr} Ducasse, R., Rossi, L., {\it Blocking and invasion for reaction-diffusion equations in periodic media}, Calc. Var. Part. Diff. Equations 57 (2018), no. 5.

\bibitem{dgm} Ducrot, A., Giletti, T., Matano, H., {\it Existence and convergence to a propagating terrace in one-dimensional reaction-diffusion equations}, Trans. Amer. Math. Soc. 366 (2014), 5541-5566.

\bibitem{ducrot} Ducrot, A., {\it A multi-dimensional bistable nonlinear diffusion equation in a periodic medium}, Math. Ann. 366 (2016), no. 1-2, 783-818.

\bibitem{fg} Gartner, J., Freidlin, M. I., {\it The propagation of concentration waves in periodic
and random media}, Dokl. Akad. Nauk SSSR 249 (1979), no. 3, 521-525.

\bibitem{fz} Fang, J., Zhao, X.-Q.,  {\it Bistable traveling waves for monotone semiflows with applications}, J. Eur. Math. Soc. 17 (2015), 2243-2288.

\bibitem{fm1} Fife, P. C.,  McLeod, J. B.,  {\it The approach of solutions of nonlinear diffusion equations to travelling front solutions},  Arch. Ration. Mech. Anal. {\bf 65} (1977), 335-361.

\bibitem{fm2}  Fife, P. C.,  McLeod, J. B.,  {\it A phase plane discussion of convergence to traveling fronts for 
nonlinear diffusion},  Arch. Ration. Mech. Anal. {\bf 75} (1981),  281-314.  

\bibitem{gm} Giletti, T., Matano, H.,  {\it Existence and uniqueness of propagating terraces}, Commun. Contemp. Math., to appear.

\bibitem{gr} Giletti, T., Rossi, L., {\it Pulsating solutions for multidimensional bistable and multistable equations}, Math. Ann., to appear. 

\bibitem{giletti-rossi-wip} Giletti, T., Rossi, L., work in preparation.

\bibitem{h} Hamel, F. {\it Qualitative properties of monostable pulsating fronts: exponential decay and monotonicity}, J. Math. Pures Appl. (9) 89 (2008), no. 4, 355-399.

\bibitem{hmr} Hamel, F., Monneau, R., Roquejoffre, J.-M. {\it Existence and qualitative properties of multidimensional conical bistable fronts}, Disc. Cont. Dyn. Systems A 13 (2005), 1069-1096.


\bibitem{nadin} Nadin, G. {\it The effect of the Schwarz rearrangement on the periodic principal eigenvalue of a nonsymmetric operator}, SIAM J. Math. Anal. 4 (2010), 2388-2406.

\bibitem{nt} Ninomiya, H., Taniguchi, M. {\it Existence and global stability of traveling curved fronts in the Allen--ahn equations},
J. Differential Equations 213 (2005), no. 1, 204-233.

\bibitem{po1} Pol{\'a}{\v c}ik, P.,  {\it Planar propagating terraces and the asymptotic one-dimensional symmetry of solutions of semilinear parabolic equations}, SIAM J. Math. Anal. 49 (2017), 3716-3740.

\bibitem{po2} Pol{\'a}{\v c}ik, P.,  {\it Propagating terraces and the dynamics of front-like solutions of reaction-diffusion equations on $\R$}, Mem. Amer. Math. Soc. 264 (2020), no. 1278, v+87 pp.

\bibitem{rossi} Rossi, L., {\it The Freidlin-Gartner formula for general reaction terms}, Adv. Math. 317 (2017), 267-298.

\bibitem{skt} Shigesada, N., Kawasaki, K., Teramoto, E., {\it Traveling periodic waves in heterogeneous environment}, Theoret. Population Biol. 30 (1986), no. 1, 143-160.  

\bibitem{xin} Xin, J., {\it  Existence and stability of traveling waves in periodic media governed by a bistable nonlinearity}, J. Dyn. Diff. Eq. 3 (1991), 541-573.

\bibitem{xin2} Xin, J., {\it  Existence and nonexistence of traveling waves and reaction-diffusion front propagation in periodic media}, J. Statist. Phys. 73 (1993), no. 5-6, 893-926.
\end{thebibliography}
\end{document}